\documentclass[11pt,times,twoside]{article}

\usepackage{graphicx,latexsym,euscript,makeidx,color,bm}
\usepackage{amsmath,amsfonts,amssymb,amsthm,thmtools,mathtools,mathrsfs,enumerate}
\usepackage[colorlinks,linkcolor=blue,anchorcolor=green,citecolor=red]{hyperref}

\usepackage{geometry}
\newcommand{\abs}[1]{\left\vert#1\right\vert}
\newcommand{\set}[1]{\left\{#1\right\}}
\newcommand{\ul}{\underline}

\newcommand{\uly}{\underline{Y}}
\newcommand{\ulz}{\underline{Z}}

\def\5n{\negthinspace \negthinspace \negthinspace \negthinspace \negthinspace }
\def\4n{\negthinspace \negthinspace \negthinspace \negthinspace }
\def\3n{\negthinspace \negthinspace \negthinspace }
\def\2n{\negthinspace \negthinspace }
\def\1n{\negthinspace }
   \def\cA{{\cal A}}

\def\dbE{\mathbb{E}}       \def\IE{\mathbb{E}}
\def\dbF{\mathbb{F}} \def\sF{\mathscr{F}}      
\def\dbG{\mathbb{G}}         
\def\dbH{\mathbb{H}}   \def\cH{{\cal H}}  
     
   \def\cJ{{\cal J}}  
   \def\cK{{\cal K}}

\def\dbN{\mathbb{N}}       \def\IN{\mathbb{N}}
     
       \def\IP{\mathbb{P}}
       \def\IQ{\mathbb{Q}}
\def\dbR{\mathbb{R}}       \def\IR{\mathbb{R}}

\def\Om{\Omega}

\def\ss{\smallskip}                
\def\ms{\medskip}                
               
\def\ds{\displaystyle}

\def\no{\noindent}        \def\q{\quad}                      
\def\ns{\noalign{\ss}}    \def\qq{\qquad}                    
    \def\hb{\hbox}                     
                   
         \def\rf{\eqref}                    
  \def\deq{\triangleq}               
            \def\({\Big (}
\def\les{\leqslant}                  \def\){\Big )}
\def\ges{\geqslant}       \def\esssup{\mathop{\rm esssup}}   \def\[{\Big[}
           \def\]{\Big]}
                   \def\cd{\cdot}

                 \def\o{\omega}

    \def\t{\tau}     \def\f{\varphi}  \def\i{\infty}   

\def\bde{\begin{definition}\label}    \def\ede{\end{definition}}
\def\be{\begin{equation}}
\def\bel{\begin{equation}\label}      \def\ee{\end{equation}}
\def\bt{\begin{theorem}\label}        \def\et{\end{theorem}}
\def\bc{\begin{corollary}\label}      \def\ec{\end{corollary}}
\def\bl{\begin{lemma}\label}          \def\el{\end{lemma}}
\def\bp{\begin{proposition}\label}    \def\ep{\end{proposition}}
\def\bas{\begin{assumption}\label}    \def\eas{\end{assumption}}
\def\br{\begin{remark}\label}         \def\er{\end{remark}}
\def\bex{\begin{example}\label}       \def\ex{\end{example}}
\def\ba{\begin{array}}                \def\ea{\end{array}}
\def\ben{\begin{enumerate}}           \def\een{\end{enumerate}}

\newtheorem{theorem}{Theorem}[section]
\newtheorem{definition}[theorem]{Definition}
\newtheorem{proposition}[theorem]{Proposition}
\newtheorem{corollary}[theorem]{Corollary}
\newtheorem{lemma}[theorem]{Lemma}
\newtheorem{remark}[theorem]{Remark}

\newtheorem{assumption}[theorem]{Assumption}
\newtheorem{example}[theorem]{Example}

\makeatletter
   
   \@addtoreset{equation}{section}
\makeatother

\allowdisplaybreaks[4]

\begin{document}

\title{{\bf Large Deviation Principle for Backward Stochastic Differential Equations with a stochastic Lipschitz condition on $z$ }
}

\author{Yufeng Shi\thanks{Institute for Financial Studies and School of Mathematics, Shandong University, Jinan 250100, Shandong, China (Email: {\tt yfshi@sdu.edu.cn}).
This author is supported by National Key R\&D Program of China (Grant No. 2018YFA0703900), the National Natural Science Foundation of China (Grant Nos. 11871309 and 11371226)}~,~~~
Jiaqiang Wen\thanks{ Department of Mathematics, Southern University of Science and Technology, Shenzhen, Guangdong, 518055, China (Email: {\tt wenjq@sustech.edu.cn}).
This author is supported by National Natural Science Foundation of China (Grant No. 12101291) and SUSTech start-up fund (Grant No. Y01286233).}~,~~~
Zhi Yang\thanks{Corresponding author. Institute for Financial Studies, Shandong University, Jinan 250100, Shandong, China (Email: {\tt yzyyss85@163.com}).
This author is supported by the State Scholarship Fund from the China Scholarship Council (No. 201906220089). The main work was carried out during his visits to Canada in 2020 and to Shenzhen in 2021. The warm hospitality of Carleton University and SUSTech is gratefully acknowledged.}}

\maketitle


\no\bf Abstract. \rm
In this paper, a probabilistic interpretation for the viscosity solution of a parabolic partial differential equation is obtained by virtue of the solution of a class of quadratic backward stochastic differential equations (BSDEs, for short). Furthermore, we prove the convergence and the large deviation principle for the solution of this class of quadratic BSDEs, which is associated with a family of Markov processes with the diffusion coefficients that tend to be zero.

\ms

\no\bf Key words: \rm
quadratic BSDEs, parabolic PDEs, Feynman-Kac formula, Malliavin calculus, large deviation principle.

\ms

\no\bf AMS subject classifications: \rm 60H10; 60F10

\section{Introduction}

The following general {\it backward stochastic differential equation} (BSDE, for short) on a finite horizon $[0,T]$:
\bel{21.10.1} Y_t=\xi+\int_t^Tf(s,Y_s,Z_s)ds-\int_t^TZ_sdW_s,\ee
was first introduced by Pardoux and Peng \cite{ref5} in 1990, where the random variable $\xi$ is called the terminal value and the process $f$ is called the generator. A little earlier, the linear BSDE was studied by Bismut \cite{ref7}, where the author treated the linear BSDE as an adjoint process of state equation in the stochastic optimal control problem. Since then, the interest in BSDE has increased regularly, due to its important applications in mathematical finance, stochastic optimal control, and partial differential equations.
We refer the readers to Duffie and Epstein \cite{ref10}, El Karoui, Peng, and Quenez \cite{ref8}, El Karoui, Hamadne, and Matoussi \cite{ref6}, Zhang \cite{ref26}, and the references cited therein.
In order to broaden the applications of BSDEs, many efforts have been made to relax the assumptions on the {\it generator} $f$ of the BSDE \rf{21.10.1} for the existence and/or uniqueness of adapted solutions.
For example, Kobylanski \cite{ref25} proved the existence, uniqueness, and a minimal (and maximal) solution of the adapted solution of BSDE under the condition that the terminal value $\xi$ is bounded and the generator $f$ is quadratic growth with respect to $z$.
For convenience, hereafter, by a {\it quadratic BSDE}, we mean that in BSDE \rf{21.10.1}, the map $z\mapsto f(t,y,z)$ grows superlinearly, but no more than quadratically. Recently,
Briand and Hu \cite{ref22,Briand-Hu-08} obtained the existence and uniqueness of quadratic BSDE with an unbounded terminal condition;
Briand and Elie \cite{ref23} used a smart construction to obtain the existence and uniqueness of quadratic BSDEs with or without delay and bounded terminal condition.

\ms

On the one hand, as we all know, the contribution of the Feynman-Kac formula is that it makes a connection between a stochastic differential equation and a linear parabolic (or elliptic) partial differential equation. By virtue of solutions of BSDEs, Pardoux and Peng \cite{ref4} gave the probabilistic interpretation for the viscosity solutions of PDEs as the generalized Feynman-Kac formula.
On the other hand, the large deviation principle (LDP, for short) characterizes the limiting behavior of probability measure in terms of rate function, which is a very active field in applied probability and is largely used in rare events simulation. Recently, there has been a growing literature on studying the applications of LDP in finance (see \cite{ref14} and \cite{ref13}).
As far as we know, Ma and Zajic \cite{ref17} first considered the large deviation principle for the adapted solutions to the coupled forward and backward stochastic differential equations (FBSDEs, for short). Meanwhile, using the probability methods, Rainero \cite{ref15} considered the small random perturbation for BSDEs and obtained Freidlin-Wentzell's large deviation result. Subsequently, Essaky \cite{ref16} investigated the large deviation for BSDEs with the subdifferential operators. Chen and Xiong \cite{ref18} established a large deviation principle of the Freidlin and Wentzell type under the corresponding nonlinear probability for diffusion processes with small diffusion coefficients.
Some other recent works concerning LDP for BSDEs can be found in \cite{ref32,ref27,ref1}, etc.

\ms

Inspired by the literature mentioned above, in this paper, we consider the large deviation principle for the solution of a class of quadratic BSDEs. We prove that the solution of quadratic BSDEs not only converges to the solution of an ordinary differential equation but also satisfies the large deviation principle (\autoref{BSDELDP}). Note that one can directly use the contraction principle to get the large deviation principle when dealing with FBSDEs.
However, the techniques that we used here to show the large deviation principle are different from the previous works. In fact, due to the lack of a prior estimate for the solution, the method we used here is based on an approximation technique that concerning the viscosity solutions of the associated PDEs.
In detail, we first use the approximation techniques (\autoref{BSDEfbijin}) to overcome the difficulty caused by the assumption that $f$ is only continuous in $z$ and then using Malliavin calculus (\autoref{BSDEbijin}), we get the regular properties of the solutions of the BSDEs with the approximation coefficients. Furthermore, under the additional condition (A5) (see Section 2 below), the uniform convergence of the solutions of the corresponding approximation PDEs in every compact subset of their own domains are obtained (\autoref{PDEzuixiaojielianxu}). Finally, the solution of the BSDE is shown to be the viscosity solutions of the corresponding PDE via the stability property of viscosity solutions  (\autoref{21.10.5}). The above mentioned methods are also useful for us to study the perturbed FBSDEs and get the large deviation results in Section 4, especially \autoref{BSDEshedongguji} and \autoref{PDEshedongzuixiaojielianxu} play a key role in the proof.

\ms

The paper is organized as follows. In Section 2, we introduce some notations and general assumptions that will be used in the sequel. In Section 3, we give some useful results on FBSDEs and establish the relationship between the solution of BSDE and the viscosity solutions of the corresponding semi-linear parabolic PDE. The large deviation principle for the solutions of the quadratic BSDEs is established in Section 4.

\section{Preliminaries}

Let $(\Omega, \mathscr{F}, \mathbb{F}, \mathbb{P})$ be a complete filtered probability space on which a $d$-dimensional standard Brownian motion $\{W_t ; 0 \leqslant t<\infty\}$ is defined, where $\mathbb{F}=\left\{\mathscr{F}_{t} ; 0 \leqslant t<\infty\right\}$ is the natural filtration of the Brownian motion $W$ with $\mathscr{F}_{0}$ containing $\mathcal{N}$, the class of all $\mathbb{P}$-null sets of $\mathscr{F}$.
Denote by $|\cd|$ and $\langle\cd,\cd\rangle$ the Euclidean norm and dotproduct, respectively, throughout the paper.
Moreover, denote that $\sF_{s}^{t}=\mathcal{N}$ if $0\les s\les t$ and
$$
\sF_{s}^{t}=\sigma\left\{  W_{r}-W_{t}:t\les r\les s\right\}
\vee\mathcal{N},\q~ \text{if }s>t.
$$
Let us introduce some notations and spaces that will be used below.
For $k\in\dbN$ and Euclidean spaces $\dbH$ and $\dbG$, denote by
$C^{k}_{b}(\dbH,\dbG)$ the set of functions of class $C^k$ from $\dbH$ to $\dbG$ whose partial
derivations of order less than or equal to $k$ are bounded.
For every $p>1$, define
$$\ba{ll}
\ns\ds L^p_{\sF_T}(\Om;\dbH)=\Big\{\xi:\Om\to\dbH\bigm|\xi\hb{ is $\sF_T$-measurable, }\|\xi\|_{L^p}\deq\big(\dbE|\xi|^p\big)^{1\over p}<\i\Big\},\\
\ns\ds L_{\sF_T}^\i(\Om;\dbH)=\Big\{\xi:\Om\to\dbH\bigm|\xi\hb{ is $\sF_T$-measurable, }
\|\xi\|_{\i}\triangleq\esssup_{\o\in\Om}|\xi(\omega)|<\i\Big\}.\ea$$
Let $T>0$ be a fixed terminal time, and for any $t\in[0,T)$, define
\begin{align*}
\ds L_\dbF^p(t,T;\dbH)\1n=&\ \1n\Big\{\f:[t,T]\1n\times\1n\Om\to\dbH\bigm|\f \hb{ is
$\dbF$-progressively measurable, }\\
\ns\ds&\qq
\|\f \|_{L_\dbF^p(t,T)}\deq\1n\(\dbE\int^T_t\1n|\f_s|^pds\)^{1\over p}\1n<\2n\i\Big\},\\
\ns\ds L_\dbF^\infty(t,T;\dbH)=&~
\Big\{\f:[t,T]\times\Om\to\dbH\bigm|\f \hb{ is $\dbF$-progressively measurable, }\\
\ns\ds&\qq
\|\f \|_{L_\dbF^\infty(t,T)}\deq\esssup_{(s,\o)\in[t,T]\times\Om}|\f_s(\o)|<\i\Big\},\\
\ns\ds S_\dbF^p(t,T;\dbH)=&~
\Big\{\f:[t,T]\times\Om\to\dbH\bigm|\f \hb{ is
$\dbF$-adapted, continuous, }\\
\ns\ds&\qq
\|\f \|_{S_\dbF^p(t,T)}\deq\Big\{\dbE\[\sup_{s\in[t,T]}|\f_s|^p\]\Big\}^{\frac{1}{p}}<\i\Big\},\\
\ns\ds S_\dbF^\infty(t,T;\dbH)=&~
\Big\{\f:[t,T]\times\Om\to\dbH\bigm|\f \hb{ is $\dbF$-progressively measurable, continuous }\\
\ns\ds&\qq
\|\f \|_{S_\dbF^\infty(t,T)}\deq\esssup_{(s,\o)\in[t,T]\times\Om}|\f_s(\o)|<\i\Big\},\\
\cH^p_{BMO}[t,T]=&~\Big\{\f \in L^p_\dbF(t,T;\dbH)\Bigm|\|\f\|_{BMO}\equiv\sup_{ t \leqslant \t  \leqslant T}\Big\|
\dbE_\t\[(\int_\t^T|\f_s|^2ds)^{\frac{p}{2}}\] \Big\|_\i^{1\over p}<\i\Big\}.
\end{align*}
where $\t$ is the $\dbF$-stopping time and $\dbE_\t$ is the conditional expectation given $\sF_{\t}$.
Denote by $C([0,T],\dbH)$ the set of continuous functions on the interval $[0,T]$ with values in $\dbH$ and consider in this space the following uniform norm
\begin{align*}\label{uniformmetric}
\rho_{[0,T]}(\phi)\deq\sup_{t\in[0,T]}|\phi(t)|.
\end{align*}
Denote by $H([t,T],\dbH)$ the space of elements come from $\phi\in C([t,T],\dbH)$ such that there exists a square integrable function $\dot\phi$ satisfying
$$\phi_s= \phi_t+\int_t^s \dot\phi_r d r,\q~0\les t\les s\les T.$$
In other word, $\phi$ is absolutely continuous.

\ms

Next, let us recall the notion of derivation on Wiener space.
The following definition comes from Pardoux and Peng \cite{ref4},
and for more information about Malliavin calculus we refer the readers to  Nualart \cite{ref28}.
Denote by $\mathbb{S}$ the set of random variable $\xi$ of the form:
\begin{equation*}
  \xi=\varphi(W(h^{1}),...,W(h^{n})),
\end{equation*}
where $\varphi \in C^{\infty}_{b}(\mathbb{R}^{n},\mathbb{R})$, $h^{i}\in L^{2}([0,T];\mathbb{R}^{d})$ and $W(h^{i})= \int_0^T h^{i}_t dW_t$ is the Wiener integral for $i=1,...,n$.
To such a random variable $\xi$, we associate a ``derivative process" $\{ D_t \xi;\ t\in[0,T] \}$ defined as
\begin{equation*}
 D_t \xi\deq\sum_{i=1}^{n}\partial_{x_i} \varphi \big(W(h^{1}),...,W(h^{n})\big)h^{i}_t, \q~ t\in[0,T].
\end{equation*}
For $\xi \in \mathbb{S}$, we define its $1,2$-norm by:
\begin{equation*}
  \| \xi \|_{1,2}^{2}\deq\mathbb{E}\left[|\xi|^{2} + \int_0^T |D_{t}\xi|^{2} dt\right].
\end{equation*}
It can be shown (see Nualart \cite{ref28}) that the operator $D$ has a closed extension to the space $\mathbb{D}^{1,2}$, the closure of $\mathbb{S}$
with respect to the norm $\| \cdot \|_{1,2}^{2}$. Note that if $\xi \in \mathbb{D}^{1,2}$ is $\sF_{s}^{t}$-measurable, then we have $D_r \xi = 0$ for $r \in [0,T]\backslash (t,s]$.

\ms

We consider the following decoupled FBSDEs:
\begin{equation}\label{SDE}
  X^{t,x}_s = x + \int^s_t b(r,X^{t,x}_r) dr + \int_t^s \sigma (r) dW_r, \q~    t \les s \les T,
\end{equation}
as a convention, if $0 \les s \les t, \ \  X^{t,x}_s = x.$
\begin{equation}\label{BSDE}
  Y^{t,x}_s = g(X^{t,x}_T) + \int^T_s f(r,Y_r^{t,x},Z_r^{t,x})dr - \int_s^T Z_r^{t,x} dW_r, \q~    t \les s \les T.
\end{equation}

\ms

For the coefficients $b:[0,T]\times\IR^m \to \IR^m$ and $\sigma:[0,T] \to \IR^{m\times d}$ of SDE \rf{SDE}, we give the following assumption.
\begin{itemize}
  \item [(A1)] The coefficients $b$ and $\sigma$ are Borel measurable and continuous bounded with respect to $t$ and $x$. Moreover, there is a positive constant $L$ such that for all $t\in[0,T]$, $x,x'\in\dbR^m$,
  \begin{align*}
  \ds | b(t,x)|+\sum_{i=1}^{d}|\sigma_{i}(t)| \les L, \q |b(t,x)-b(t, x')| \les L|x-x'|,
  \end{align*}
  where $\sigma_{i}$ denote the $i$th column of matrix $\sigma$.
\end{itemize}
A standard argument for SDE with Lipschitz condition implies  that, under the assumption (A1), the existence result, uniqueness result and some useful estimates to the solution of SDE \rf{SDE} hold. The following result comes from Theorem 2.1 of Kunita \cite{ref3}.

\begin{lemma}\label{SDEguji}
Under the assumption (A1), given $p\ges2$, for any $(t,x)\in[0,T]\times \IR^m$,
SDE \eqref{SDE} admits a unique strong solution $X^{t,x} \in S_\dbF^{p}(0,T;\dbR^m)$. Moreover, there exists a positive constant $C_p$ depending only on $p$, $T$ and $L$ such that for all
$t,t'\in[0,T]$, $x,x'\in\dbR^m$, we have
\begin{equation*}
  \IE\left[\sup_{s\in [0,T]} |X_s^{t,x}-X_s^{t',x'}|^p \right]
  \les C_p \left(|x-x'|^p + (1+|x|^p+|x'|^p)|t-t'|^\frac{p}{2}\right).
\end{equation*}
\end{lemma}

\ms

On the other hand, for the coefficients $g:\IR^m\to\IR$ and $f:[0,T]\times \IR\times \IR^d\to \IR$ of BSDE \rf{BSDE},
 we present the following assumptions.
\begin{itemize}
  \item [(A2)] The coefficients $g$ and $f$ are both Borel measurable and continuous with respect to their items, respectively.
  \item [(A3)] There exists a positive constant $L$ such that for all $t\in[0,T]$ $x,x'\in\dbR^m$ $y,y'\in\dbR$ and $z\in\dbR^d$,
   \begin{align*}
  \ds |g(x)|  \les&~ L, \q |f(t,y,z)|  \les L(1+|y|+|z|^2),\\
  \ns\ds |g(x)-g(x')| \les&~ L|x-x'|,\q |f(t,y,z)-f(t,y',z)|  \les L|y-y'|.
   \end{align*}
  \item [(A4)] There is a positive constant $L_z$ such that for all $t\in[0,T]$, $y\in\dbR$ and $z,z'\in\dbR^d$,
      $$|f(t,y,z)-f(t,y,z')| \les L_z(1+|z|+|z'|)|z-z'|.$$
  \item [(A5)] There is a bounded subset $\mathcal{C}^0 \subseteq \IR \times \IR^d$ and a positive progressively measurable stochastic process $L_t \in \cH^2_{BMO}[0,T]$ such that $d\IP \otimes dt$-a.e., for all $(y,z) \in \mathcal{C}^0$ and $(y,z')\in \mathcal{C}^0$,
      $$  |f(t,y,z)-f(t,y,z')| \les L_t|z-z'|.$$
\end{itemize}
\begin{remark}
The bound of the subset $\mathcal{C}^0$ depends on the constants $L$ and $T$, and Assumption (A5) is weaker than the stochastic Lipschitz condition that for all $(y,z) $ and $(y,z')\in \IR \times \IR^d$ (see Section 2 of Dos Reis \cite{ref91}).
%
\end{remark}

\ms

Throughout our paper, to avoid the additional Malliavin regularization technicalities (see Cheridito and Nam \cite{ref35} for more detail), we only consider the deterministic generator $f$ (and which does not depend on $x$) case.

\ms

From Kobylanski \cite{ref25} and Chapter 7 of Zhang \cite{ref26}, we see that under the assumptions
(A2)-(A4), BSDE \rf{BSDE} with quadratic growth in $Z$ has a unique adapted solution. The assumption (A5) will be used to get the large deviation principle and the stability of the viscosity solution of related PDE later.

\section{BSDEs and semi-linear parabolic PDEs}

In this section, we first study some properties of BSDE \rf{BSDE} with quadratic growth and with Malliavin calculus. Then, the relationship between the solutions of BSDE \rf{BSDE} and the viscosity solutions of the corresponding semi-linear parabolic PDEs is established.

\subsection{Regularities of solution}


\begin{proposition}\label{BSDEbounded}
Under Assumptions (A1)-(A4), for any $(t,x) \in[0,T)\times \IR^m$, the unique adapted solution
$X^{t,x} \in S_\dbF^2(0,T;\dbR^m)$ of SDE \rf{SDE} and $(Y^{t,x} ,Z^{t,x} )\in S_\dbF^\i(0,T;\dbR) \times \cH^2_{BMO}[0,T]$ of BSDE \rf{BSDE} have the following properties: The Malliavin derivative of $X^{t,x} $ exists and bounded, i.e., $DX^{t,x}  $ is bounded; the solution $Z^{t,x} $ is bounded, i.e., $Z^{t,x} \in L_\dbF^\infty(0,T;\dbR^d)$.
\end{proposition}

\begin{proof}
First, note that under Assumption (A1), by \autoref{SDEguji}, SDE \eqref{SDE} has a unique solution $X^{t,x} \in S_\dbF^2(0,T;\dbR^m)$. Under Assumptions (A2)-(A4), by Theorem 7.3.3 of Zhang \cite{ref26}, BSDE \eqref{BSDE} has a unique solution $(Y^{t,x} ,Z^{t,x} )\in S_\dbF^\i(0,T;\dbR) \times \cH^2_{BMO}[0,T]$.

\ms

Next, under Assumption (A1), we see that by Lemma 4.2 of Cheridito and Nam \cite{ref35}, the derivative of $X^{t,x} $ exists and is bounded by some constant (denote by $M_1$), which depends only on $T $ and $L$.

\ms

Furthermore under Assumptions (A2)-(A4), by Proposition 4.3 of Cheridito and Nam \cite{ref35}, we can get that the solution $(Y^{t,x} ,Z^{t,x} )$ are adapted processes and uniformly bounded by constant $M $, which does not depend on $t$, $x$ and the constant $L_z$ of (A4).
\end{proof}

\begin{remark}
It should be pointed out that, similar to Lemma 2.1 of Briand and Elie \cite{ref23}, the boundness $M$ of $Y^{t,x} $ and $Z^{t,x} $ is not depend on the constant $L_z$ of (A4). This is important for us to obtain the uniformly bounded solutions of the BSDEs with the approximation coefficients in Proposition \ref{BSDEbijin}.
\end{remark}

Lepeltier and San Martin \cite{ref30} presented a nice approximation of classical BSDE with continuous coefficients. So it is easy to think of the similar result of Lepeltier and San Martin \cite{ref30} in the circumstance of quadratic BSDE. To the best of our knowledge, the following results
have not appear in the previous literature, so we give the detailed proof here, which is useful for us to study BSDE \rf{BSDE} with quadratic growth later.

\begin{lemma}\label{BSDEfbijin}
Let $f:[0,T]\times \IR\times \IR^d\to \IR$ be a continuous function and there exists a positive constant $L$ such that for all $(t,y,z)\in[0,T] \times \IR \times \IR^d$,
\begin{align*}
&|f(t,y,z)|\les L\big( 1+|y|+|z|^2 \big),
\end{align*}
and for all $t\in[0,T]$, $y,y'\in\IR$ and $z\in\IR^d$,
$$|f(t,y,z)-f(t,y',z)| \les L|y-y'|.$$
Then the sequence of functions
\begin{equation}
\label{f-n}
  f_n(t,y,z)\deq\inf_{v \in \IR^d}\left\{f(t,y,v)+n|z-v|^2\right\},
\end{equation}
is well defined for $n \ges 2L$. Moreover, it satisfies
\begin{itemize}
  \item [{\rm (1)}] For any $t\in[0,T]$, $y\in\IR$ and $z\in\IR^d$,
$ |f_n(t,y,z)| \les L(1+|y|+2|z|^2)$.
  \item [{\rm (2)}] For any $t\in[0,T]$, $y\in\IR$ and $z\in\IR^d$,
$f_n(t,y,z)$ is increasing in $n$.
  \item [{\rm (3)}] If $z_n \to z$ as $n\to\infty$, then
$f_n(t,y,z_n)\to f(t,y,z)$.
  \item [{\rm (4)}] For any $t\in[0,T]$, $y_1,y_2\in\IR$ and $z_1,z_2\in\IR^d$ and for sufficiently large $n$,
$$|f_n(t, y_1,z_1)-f_n(t, y_2,z_2)|\les L
|y_1-y_2|+ n(1+|z_1|+|z_2|)|z_1-z_2|.$$
\end{itemize}
\end{lemma}

\begin{proof}
We only consider the case of the generator $f$ independent of $t$ and $y$, i.e., $f(\cd)=f(z)$, and the case of $f(t,y,z)$ could be proved similar without substantial difference. Let $h:\IR^d\to \IR$ be a continuous function and $|h(z)|\les L\big( 1+|z|^2 \big)$, then for any $n \ges 2L$, we define
$$h_n(z)\deq \inf_{v\in \IR^d}\left\{h(v)+n|z-v|^2\right\}.$$
From the definition, we see that $h_n(z) \les h(z) \les L(1+|z|^2)$ and,
\begin{align*}
\ns\ds h_n(z)  & \ges \inf_{v\in \IR^d}\left\{-L -L|v|^2 +n|z-v|^2\right\}  \\
\ns\ds & = \inf_{v\in \IR^d}\left\{-L -L|v|^2 +2L |z-v|^2+(n-2L)|z-v|^2\right\} \\
\ns\ds & \ges \inf_{v\in \IR^d}\left\{-L -2L|z|^2 +(n-2L)|z-v|^2\right\} \\
\ns\ds & = -L -2L|z|^2.
\end{align*}
So $|h_n(z)| \les L(1+2|z|^2)$, from which item (1) holds.
Moreover, item (2) comes from the definition of $h_n(\cd)$ directly.

\ms

Now we prove the item (3), and consider $z_n \to z$, then for every $n$, there exists $v_n \in \IR^d$ such that
\begin{align*}
\ds h(z_n) \ges h_n(z_n) &  \ges h(v_n)+n\abs{z_n-v_n}^2-\frac{1}{n} \\
\ns\ds & \ges -L -L|v_n|^2 +n\abs{z_n-v_n}^2-\frac{1}{n} \\
\ns\ds & \ges -L -2L|z_n|^2 +(n-2L)\abs{z_n-v_n}^2-\frac{1}{n}.
\end{align*}
Since ${h(z_n)}$ is bounded, we deduce that $\lim\limits_{n\to\infty}\sup (n-2L)\abs{z_n-v_n}^2 <\infty$. In particular, when $v_n \to z$, we have
$$\lim_{n\to\infty}\sup (n-2L)\abs{z_n-z}^2 <\infty.$$
Moreover, we have that
 $$h(z_n) \ges h_n(z_n)   \ges h(v_n) -\frac{1}{n},$$
from which $ h_n(z_n) \rightarrow h(z)$ when $z_n \to z$, this implies that item (3) holds.

\ms

In order to prove item (4), for any $z \in \IR^d$, we take $\epsilon > 0$ and consider $v_{\epsilon} \in \IR^d$  (in fact for sufficiently large $n$, we can assumed that $\abs{z-v_{\epsilon}} \les \frac{1}{2}$) such that
\begin{align*}
\ns\ds h_n(z)  & \ges h(v_{\epsilon})+n\abs{z-v_{\epsilon}}^2-\epsilon  \\
\ns\ds & = h(v_{\epsilon})+n\abs{v-v_{\epsilon}}^2 - n\abs{v-v_{\epsilon}}^2+n\abs{z-v_{\epsilon}}^2-\epsilon \\
\ns\ds & \ges h_n(v) - n\abs{v-v_{\epsilon}}^2+n\abs{z-v_{\epsilon}}^2-\epsilon \\
\ns\ds & \ges h_n(v) - n\abs{z-v}(\abs{z-v_{\epsilon}}+\abs{v-v_{\epsilon}})-\epsilon \\
\ns\ds & \ges h_n(v) - n\abs{z-v}(1+|z|+|v|)-\epsilon.
\end{align*}
Therefore, interchanging the role of $z$ and $v$, and since $\epsilon$ is arbitrary we can get
$$|h_n(z)-h_n(v)| \les n(1+|z|+|v|)\abs{z-v}.$$
Finally, for any $y_1,y_2\in\IR$, $z\in\IR^d$ and $t\in[0,T]$, by the inequality
$$\left|\inf_{v \in \IR^d}f(v)- \inf_{v \in \IR^d}g(v)\right| \les \sup_{v \in \IR^d} \left| f(v)-g(v) \right|,$$
and the Lipschitz condition of $f$ with respect to the spatial variable $y$, we get that
$$\abs{f_n(t, y_1,z)-f_n(t, y_2,z)}\les L\abs{y_1-y_2},$$
from which the desired result follows.
\end{proof}

\begin{remark}\label{smooth}
The idea of the proof of the above lemma is inspired by Moreau--Yosida's regularization of convex optimization (see Barrieu and El Karoui \cite{ref40}).
Indeed, on the other hand, we can also construct a smooth approximation to the generator $f(t,y,z)$ of BSDE \rf{BSDE}. For instance, we could denote $v = (y,z)$, $v' = (y',z') \in \IR\times \IR^d$ and let $\eta \in C^{\infty} ( \IR\times \IR^d)$ be the mollifier as below,
$$
\eta(v) \deq
\left\{
\begin{array}{ll}
\ns\ds C  \exp \left(- \frac{1}{1- |v|^2}\right),\q &\text { if } |v|<1; \\
\ns\ds 0,\q &\text{ otherwise},
\end{array} \right.
$$
where $C$ is a positive constant which could be selected such that
$$\int_{\mathbb{R}^{1+d}} \eta(v) dv = 1.$$
Then, for any positive integer $n$, we can set
$$\eta_n(v) \deq n^{1+d}\eta(n v),$$
and define
$$
\tilde{f}_n(t,v)\deq \int_{\IR^{1+d}} f(t,v') \eta_n(v - v') dv'.
$$
For more detailed discussion about the convolution and smoothing, we refer the readers to Appendix C.5 of Evans \cite{ref38}.
\end{remark}

\ms

In the following, for positive number $n \geqslant 2L$, we consider the following SDE and BSDE:

\begin{equation}\label{SDEjieduan}
  X^{t,x}_s = x + \int^s_t b(r,X^{t,x}_r) dr + \int_t^s \sigma (r) dW_r, \q~    t \les s \les T;
\end{equation}
and
\begin{equation}\label{BSDEjieduan}
Y^{n,t,x}_s = g(X^{t,x}_T) + \int^T_s f_n(r,Y_r^{n,t,x},Z_r^{n,t,x})dr - \int_s^T Z_r^{n,t,x} dW_r, \q~    t \les s \les T;
\end{equation}
where $f_n(t,y,z)$ is defined by \eqref{f-n}. For the above BSDE \rf{BSDEjieduan}, we have the following proposition concerning the solutions.

\begin{proposition}\label{BSDEbijin}
Under Assumptions (A1)-(A3), for any $(t,x)\in[0,T)\times \IR^m$ and sufficiently large $n \in \dbN$,
BSDE \eqref{BSDEjieduan} admits a unique solution $(Y^{n,t,x} ,Z^{n,t,x} )\in S_\dbF^\i(0,T;\dbR) \times L_\dbF^\i(0,T;\dbR^d)$ such that $Y ^{n,t,x}$ is increasing in $n$.
Moreover, $(Y^{n,t,x},Z^{n,t,x})\rightarrow (\uly^{t,x},\ulz^{t,x})$ in the space $S_\dbF^2(0,T;\dbR) \times L_\dbF^2(0,T;\dbR^d)$ as $n\rightarrow \infty$, where $( \uly^{t,x},\ulz^{t,x})$ is the minimal solution of BSDE \eqref{BSDE} (in the sense that $\uly^{t,x} \les Y^{t,x}$ for any other possible solutions $(Y^{t,x},Z^{t,x}))$.
\end{proposition}

\begin{proof}
Under Assumptions (A1)-(A3), by \autoref{BSDEfbijin}, we know that for any sufficiently large  $n \in \dbN $, the functions $g(x)$ and $f_n(r,y,z)$ satisfy the conditions of \autoref{BSDEbounded}, so there exists a unique pair processes $(Y^{n,t,x},Z^{n,t,x}) \in S_\dbF^\i(0,T;\dbR) \times L_\dbF^\infty(0,T;\dbR^d)$ that solves BSDE \eqref{BSDEjieduan} and the bounds of the solution $(Y^{n,t,x},Z^{n,t,x})$ denote by constant $M $, which does not depend on $n$, $t$ and $x$. Furthermore, noting that $f_n(r,y,z)$ is increasing in $n$, by Theorem 7.3.1 of Zhang \cite{ref26} (the comparison theorem), we get that $Y^{n,t,x}$ is increasing with respect to $n$.

\ms

Next, by the monotone stability proposition (see Lemma 3 of Briand and Hu \cite{ref22}), the pair $(Y^{n,t,x},Z^{n,t,x})\rightarrow (\uly^{t,x},\ulz^{t,x})$ in the space $S_\dbF^2(0,T;\dbR) \times L_\dbF^2(0,T;\dbR^d)$ as $n\rightarrow \infty$. Finally, since the bounds of the solution $(Y^{n,t,x},Z^{n,t,x})$ does not depend on $n$, we get that $( \uly^{t,x},\ulz^{t,x}) \in  S_\dbF^\i(0,T;\dbR) \times L_\dbF^\infty(0,T;\dbR^d)$.
\end{proof}

\begin{remark}
In fact, $( \uly^{t,x},\ulz^{t,x})$ is the unique solution of BSDE \eqref{BSDE}, for more detail refer to Theorem 3.12 of Shi and Yang \cite{ref201}.
\end{remark}

\ms

Now, we study a prior estimate for the solutions of BSDE \eqref{BSDEjieduan}, the method used here is inspired by Proposition 2.3 of Briand and Elie \cite{ref23}.

\begin{proposition}
\label{BSDEguji}
Under Assumptions (A1)-(A3) and (A5), for any $(t,x)$, $(t',x') \in[0,T)\times \IR^m$ and sufficiently large $n \in  \dbN$, denote by $(X^{t,x},Y^{n,t,x},Z^{n,t,x})$, $(X^{t',x'},Y^{n,t',x'},Z^{n,t',x'})\in S_\dbF^2(0,T;\dbR^m) \times S_\dbF^\i(0,T;\dbR) \times L_\dbF^\infty(0,T;\dbR^d)$ the adapted solutions of FBSDE \eqref{SDEjieduan}-\eqref{BSDEjieduan}, respectively. Then, for any $p > p_0$, we have
\begin{align*}
   \|Y_s^{n,t,x}-Y_s^{n,t',x'}\|_{S_\dbF^p(0,T;\dbR)}^p
 \les C_p  (|x-x'|^p + (1+|x|^p+|x'|^p)|t-t'|^\frac{p}{2}),
\end{align*}
where $p_0$ is a positive constant that greater than 1, and $C_p$ (independent of $t,x,t',x'$ and $n$) is a positive constant.
\end{proposition}

\begin{proof}
First, under Assumptions (A1)-(A3), and by \autoref{BSDEbounded} and \autoref{BSDEbijin}, we see that for any sufficiently large $n \in  \dbN$, FBSDE \eqref{SDEjieduan}-\eqref{BSDEjieduan} admits a unique adapted solution $(X^{t,x},Y^{n,t,x},Z^{n,t,x})\in S_\dbF^2(0,T;\dbR^m) \times S_\dbF^\i(0,T;\dbR) \times L_\dbF^\infty(0,T;\dbR^d)$. Second, for simplicity presentation, denote by $M$ the bounds of the solution $(Y^{n,t,x},Z^{n,t,x})$, which does not depend on the parameters $n$, $t$ and $x$ (see \autoref{BSDEbijin}). We also would like to let $m=d=1$ for simplicity discussion, and the case of multi-dimensions can be proved using the similar methods without substantial difficulties.

\ms

For $(t,x)$, $(t',x') \in[0,T)\times \IR$ with $t \ges t'$, denote
$$\Delta Y^n\deq Y^{n,t,x}-Y^{n,t',x'},\q \Delta Z^n \deq Z^{n,t,x}-Z^{n,t',x'},\q  \Delta \xi_T \deq g(X^{t,x}_T)-g(X^{t',x'}_T).$$
By the classical linearization argument, we have that
\begin{align}\label{LBSDE}
  \Delta Y^n_s = \Delta \xi_T &+  \int^T_s (a^n_r \Delta Y^n_r + b^n_r \Delta Z^n_r)dr
   - \int_s^T \Delta Z^n_r dW_r, \q~    t \les s \les T,
\end{align}
where the processes $a^n$ and $b^n$ are resoectively defined by
\begin{equation*}
a^n_r \deq \left\{
\begin{array}{cc}
\frac{f_n(r,Y_r^{n,t,x},Z_r^{n,t,x})-f_n(r,Y_r^{n,t',x'},Z_r^{n,t,x})}{\Delta Y^n_r}, & \text{ if }  \Delta Y^n_r \neq 0;\\
0, & \text{ if }   \Delta Y^n_r = 0,
\end{array}
\right.
\end{equation*}

\begin{align*}
b^n_r \deq \left\{
\begin{array}{cc}
  \frac{f_n(r,Y_r^{n,t',x'},Z_r^{n,t,x})-f_n(r,Y_r^{n,t',x'},Z_r^{n,t',x'})}{\Delta Z^n_r}, & \text{if }   \Delta Z^n_r \neq 0;\\
0, & \text{if }  \Delta Z^n_r = 0.
\end{array}
\right.
\end{align*}
From Assumption (A3) and \autoref{BSDEfbijin}, we have
$$|a^n_r| \les L\q \text{and}\q |b^n_r| \les n(1+|Z_r^{n,t,x}|+|Z_r^{n,t',x'}|).$$
For a bounded subset $\mathcal{C}_0 \subseteq \IR \times \IR^d$,
using the same arguments as in \autoref{BSDEfbijin}, there exists an integer $N_0 >0$ such that for any $n \ges N_0$ and $(y,z_1)$ and $(y,z_2)\in \mathcal{C}_0$, we have
\begin{align}\label{ineq}
&|f_n(r,y,z_1) -f_n(r,y,z_2)| \nonumber \\
& =\left|\inf_{v \in \IR^d}\set{f(r,y,v)+n\abs{z_1-v}^2} -  \inf_{v \in \IR^d}\set{f(r,y,v)+ n\abs{z_2-v}^2} \right| \nonumber \\
& =\left|\inf_{v \in \IR^d}\set{f(r,y,z_1-v)+n\abs{v}^2} -  \inf_{v \in \IR^d}\set{f(r,y,z_2-v)+ n\abs{v}^2} \right| \nonumber \\
& =\left|\inf_{|v|\les 1}\set{f(r,y,z_1-v)+n\abs{v}^2} -  \inf_{|v|\les 1}\set{f(r,y,z_2-v)+ n\abs{v}^2} \right| \nonumber \\
& \les \sup_{|v|\les 1} \set{|f(r,y,z_1-v) - f(r,y,z_2-v)|} \nonumber \\
& \les L_r |z_1 - z_2|, \qq d\IP \otimes dr\hb{-a.e.},
\end{align}
where the last inequality holds comes from Assumption (A5). In addition, note that the processes $Y^{n,t',x'}$, $Z^{n,t,x}$ and $Z^{n,t',x'}$ are bounded by $M$.
Combing the inequality \eqref{ineq}, for any $n \ges N_0$,
$$|f_n(r,Y_r^{n,t',x'},Z_r^{n,t,x})-f_n(r,Y_r^{n,t',x'},Z_r^{n,t',x'})|  \les   L_r |Z_r^{n,t,x} - Z_r^{n,t',x'}|,\qq d\IP \otimes dr\hb{-a.e.}$$
Now we can get
\begin{equation*}
  |b^n_r| \les \text{max}\{ N_0(1+|Z_r^{n,t,x}|+|Z_r^{n,t',x'}|), \ \ L_r \} \les \text{max}\{ N_0(1+2M), \ \ L_r  \},\qq d\IP \otimes dr\hb{-a.e.}
\end{equation*}
Denote
$$N^* \deq \sup_{n} \|  b^n  \|_{\cH^2_{BMO}[0,T] } < \infty.$$
From Girsanov theorem, there exists an equivalent probability measure $\mathbb{Q}$ under which the process
$$W^{b_n}\deq W - \int_0^\cdot b^n_r dr$$
is a Brownian motion. We set $e_s = \exp\{\int_0^s a^n_r dr\}$ and compute the linear BSDE \eqref{LBSDE} to obtain that
\begin{align*}
  e_s \Delta Y^n_s = e_T \Delta \xi_T
   - \int_s^T e_r \Delta Z^n_r dW^{b_n}_r, \q~   t \les s \les T.
\end{align*}
Denoting $(\mathcal{E}^{b^n}_t)_{0 \les t \les T}$ the Dol\'{e}ans-Dade exponential of $b^n$ and noting that $|a^n_r| \les L$, as a by-product of the previous equality, we get
\begin{align}\label{inLBSDE}
\ds  |\Delta Y^n_s| & \les  \mathbb{E}^{\mathbb{Q}} \left[ |\Delta \xi_T|  e_T /e_s \mid  \sF_{s}^{t} \right]
  \les  (\mathcal{E}^{b^n}_s)^{-1} \mathbb{E}\left[\mathcal{E}^{b^n}_T |\Delta \xi_T|  e_T /e_s \mid \sF_{s}^{t}\right] \nonumber \\
\ns\ds  & \les  e^{LT} (\mathcal{E}^{b^n}_s)^{-1} \mathbb{E}\[  |\Delta \xi_T|  \mathcal{E}^{b^n}_T  \mid \sF_{s}^{t}\].
\end{align}
Note that $\|  b^n  \|_{\cH^2_{BMO}[0,T]} < \infty$, so the reverse H\"{o}lder's inequality implies that
\begin{equation}\label{Rholder}
 \mathbb{E}\[(\mathcal{E}^{b^n}_T)^q | \sF_{s}^{t} \] \les C_{n,q}^{*}(\mathcal{E}^{b^n}_s)^q, \ \  t \les s \les T, \q~ 1< q < q^{*}_n ,
\end{equation}
where the constants $q^{*}_n > 1$ and $(C_{n,q}^{*})_{1< q < q^{*}_n}$ are given by
\begin{align*}
\ds & q^{*}_n\deq \phi ^{-1}(\|  b^n  \|_{\cH^2_{BMO}[0,T]}), \q \text{with}\q \phi(q) = ( 1 + \frac{1}{q^2} \log \frac{2q-1}{2(q-1)})^{1/2}-1,  \\
\ns\ds & C_{n,q}^{*}\deq 2[1-2(q-1)(2q-1)^{-1}\exp\{q^2(\|  b^n  \|^2_{\cH^2_{BMO}[0,T]}+2\|  b^n  \|_{\cH^2_{BMO}[0,T]})\}]^{-1},
\end{align*}
see the proof of Theorem 3.1 in Kazamaki \cite{ref41} for details.
Note that $q^{*}_n$ is non-increasing with respect to $\|  b^n  \|_{\cH^2_{BMO}[0,T]}$,
$(C_{n,q}^{*})_{1< q < q^{*}_n}$ is non-decreasing with respect to $\|  b^n  \|_{\cH^2_{BMO}[0,T]}$,
and $\|  b^n  \|_{\cH^2_{BMO}[0,T]} \les N^*$. Denote
\begin{equation*}
q^{*}\deq  \phi ^{-1}(N^*),  \q~ C_{q}^{*}\deq 2[1-2(q-1)(2q-1)^{-1}\exp\{q^2(N^* N^*+2N^*)\}]^{-1},
\end{equation*}
then we have $q^{*}_n \ges q^{*} >1$ and $C_{n,q}^{*} \les C_{q}^{*} $, hence the inequality \eqref{Rholder} implies
\begin{equation}\label{Rholder2}
 \mathbb{E}\[(\mathcal{E}^{b^n}_T)^q | \sF_{s}^{t} \] \les C_{q}^{*}(\mathcal{E}^{b^n}_s)^q, \q~  t \les s \les T, \q~ 1< q < q^{*},
\end{equation}
where the constants $q^{*}$ and $C_{q}^{*} $ do not depend on $n$. Denoting by $p$ the conjugate of a given $q \in (1,q^{*})$ and combing the conditional H\"{o}lder inequality together with \eqref{inLBSDE} and \eqref{Rholder2}, we have
\begin{align*}
  |\Delta Y^n_s|^p  \les & \ \ e^{pLT} (\mathcal{E}^{b^n}_s)^{-p} \( \mathbb{E}\[ \( \mathcal{E}^{b^n}_T \)^q | \sF_{s}^{t}\] \)^\frac{p}{q}
    \mathbb{E}\[|\Delta \xi_T|^p | \sF_{s}^{t}\] \\
   \les & \ \ e^{pLT} (C_{q}^{*})^\frac{p}{q}
    \mathbb{E}\[ |\Delta \xi_T|^p | \sF_{s}^{t} \], \q~ 0 \les t \les s \les T.
\end{align*}
Note that $p^*\deq q^{*}/(q^{*}-1)$ is independent of $n$, we deduce from Doob's maximal inequality that
\begin{align*}
      \| \Delta Y^n  \|_{S_\dbF^p(0,T;\dbR)}^p  \les C_p \mathbb{E}\[ |\Delta \xi_T|^p \]
       \les C_p  (|x-x'|^p + (1+|x|^p+|x'|^p)|t-t'|^\frac{p}{2}),  \q~\forall p >p^* ,
\end{align*}
where $C_p$ is a universal constant independent of $n$ and could be change from line to line. We point out that the last inequality comes from the Lipschitz condition of $g$ and \autoref{SDEguji}.
Finally, picking $p_0=p^*$ concludes the proof.
\end{proof}

\subsection{FBSDEs connection to PDEs}

In this subsection we study the relationship between the solution $( \ul{Y}^{t,x},\ul{Z}^{t,x})$ of BSDE \eqref{BSDE} and the viscosity solution of the following semi-linear parabolic PDE:
\begin{equation}\label{PDE}\left\{\begin{aligned}
\ds &\frac{\partial u}{\partial t} \left( t,x\right) + \mathcal{L}_{t,x} u \left(t,x \right)+
  f\left(t,u \left(t,x \right),\left(\sigma^T \nabla u \right) \left(t,x \right)\right)=0, \\
\ns\ds &u \left(T,x \right) = g(x),
\end{aligned}\right.
\end{equation}
where $\nabla u$ is the gradient in spatial variable $x$ (column vector) and the second order differential operator $\mathcal{L}_{t,x}$ being the infinitesimal semi-group generator of the solution $(X_s^{t,x})_{s\in[t,T]}$ of SDE \eqref{SDE} given by
$$
{\cal L}_{t,x}\deq\frac 1 2\sum_{i,j=1}^m a_{i,j}\left( t \right) \frac{\partial
^2}{\partial x_i\partial x_j}+\sum_{i=1}^m b_i\left( t,x\right) \frac
\partial {\partial x_i},  \q~ a_{i,j} = ( \sigma \sigma^T )_{i,j}.
$$
Denote by $C^{1,2}([0,T]\times \IR^m;\IR)$ the space of all functions from $[0,T]\times \IR^m$ to $\IR$ whose derivatives up to the first order with respect to time variable and up to the second order with respect to spatial variable are continuous. Recall the definition for the viscosity solution of PDE \eqref{PDE}, and we refer the readers to \cite{ref9} for details.

\begin{definition}
A continuous function $u:[0,T] \times \mathbb{R}^n\to\mathbb{R}$ is called a viscosity subsolution (resp. supersolution) of PDE \eqref{PDE}, if for any $x \in \mathbb{R}^m$,  $u(T,x) = g(x)$, and for any $\varphi \in C^{1,2}([0,T]\times \mathbb{R}^m;\mathbb{R})$, $(t,x) \in [0,T) \times \mathbb{R}^m$, $\varphi-u$ attains a minimum (resp. maximum) at $(t,x)$ and $\varphi$ satisfies
  \begin{eqnarray*}
  && \frac{\partial \varphi}{\partial t} \left( t,x\right) + \mathcal{L}_{t,x} \varphi \left(t,x \right)+
  f\left(t,u \left(t,x \right),\left(\sigma^T \nabla \varphi \right) \left(t,x \right)\right) \ges 0, \\ 
  && \big(resp. ~~\frac{\partial \varphi}{\partial t} \left( t,x\right) + \mathcal{L}_{t,x} \varphi \left(t,x \right)+
  f\left(t,u \left(t,x \right),\left(\sigma^T \nabla \varphi \right) \left(t,x \right)\right) \les 0\big).
  \end{eqnarray*}
We call $u$ the viscosity solution of the PDE \eqref{PDE} if $u$ is both a viscosity subsolution and a viscosity supersolution.

\end{definition}

\ms

The following lemma can be easily obtained by Theorem 7.3.6 of Zhang \cite{ref26}, and from which we have a corollary.

\begin{lemma}\label{BSDEgai}
Suppose that Assumptions (A1)-(A4) hold. For any $(t,x)\in[0,T)\times \IR^m$, the
function $u$ defined by $u(t,x) \deq Y^{t,x}_t$ is a viscosity solution of PDE \eqref{PDE}, where $Y^{t,x}_s$ is the solution of BSDE \eqref{BSDE}.
\end{lemma}

\begin{corollary}\label{PDEnian}
Under Assumptions (A1)-(A3), for any $(t,x)\in[0,T)\times \IR^m$ and sufficiently large $n \in  \dbN$, we have that
\begin{equation}
\label{u-n}
u^n (t,x) \deq Y^{n,t,x}_t,
\end{equation}
where $Y^{n,t,x}_s$ is the solution of BSDE \eqref{BSDEjieduan}, is a viscosity solution of the following PDE:
\begin{equation}\label{PDEjieduan}\left\{\begin{aligned}
\ds &\frac{\partial u^n}{\partial t} \left( t,x\right) + \mathcal{L}_{t,x} u^n \left(t,x \right)+
  f_n \left(t,u^n \left(t,x \right),\left(\sigma^T \nabla u^n \right) \left(t,x \right)\right)=0, \\
\ns\ds &u^n \left(T,x \right) = g(x),
\end{aligned}\right.
\end{equation}
where $f_n(t,y,z)$ is defined by (\ref{f-n}).
\end{corollary}

\ms

Obviously, from \autoref{BSDEbijin}, $u^n (t,x)$ defined in (\ref{u-n}) increasingly converges to $\underline{u} (t,x)$, where $\underline{u} (t,x)$ is defined by
\begin{equation}
\label{u-}
  \ul{u}(t,x) \deq \uly_t^{t,x},
\end{equation}
with $( \uly^{t,x},\ulz^{t,x})$ is the solution of BSDE \eqref{BSDE}.

\ms

In the following proposition, we indicate that $u^n (t,x)$ converges to $\ul{u}(t,x)$ uniformly in any compact subset of their domains.
\begin{proposition}\label{PDEzuixiaojielianxu}
Under Assumptions (A1)-(A3) and (A5), then the sequence $u^n (t,x)$ defined in (\ref{u-n}) converges to $\ul{u}(t,x)$ defined in (\ref{u-}) uniformly in every compact subset of $[0,T]\times \IR^m$, as $n \rightarrow \infty$.
\end{proposition}

\begin{proof}
Take $(t,x)\in[0,T)\times \IR^m$, for any sufficiently large $n \in  \dbN$, we consider the decoupled FBSDE \eqref{SDEjieduan}-\eqref{BSDEjieduan}
and denote by $(Y^{F^1,t,x},Z^{F^1,t,x})$ and $(Y^{F^2,t,x},Z^{F^2,t,x})$ the solutions of two BSDEs with the generator
$$F^1(r,y,z)= L(1+|y|+2|z|^2), \q~ F^2(r,y,z)= - L(1+|y|+2|z|^2),$$
respectively. By \autoref{BSDEfbijin} and the comparison theorem (Theorem 7.3.1 in Zhang \cite{ref26}), we get
$$ Y^{F^2,t,x} \les Y^{n,t,x} \les Y^{F^1,t,x}.$$
Furthermore, we have that $ w^2 (t,x) \les u^n (t,x) \les w^1 (t,x) $, where $w^i (t,x)\deq Y^{F^i,t,x}_t$ is the viscosity solution of the following PDE
\begin{equation*}\left\{\begin{aligned}
\ds   &\frac{\partial w^i}{\partial t} \left( t,x\right) + \mathcal{L}_{t,x} w^i \left(t,x \right)+
  F^i \left(t,w^i \left(t,x \right),\left(\sigma^T \nabla w^i \right) \left(t,x \right)\right)=0, \\
\ns\ds  &w^i \left(T,x \right) = g(x),
\end{aligned}\right.
\end{equation*}
respectively for $i=1,2$.
So for any compact set $\mathcal{C}$ of $[0,T]\times \IR^m$, the sequence $u^n $ defined in (\ref{u-n}) is uniformly bounded on $\mathcal{C}$.

\ms

Moreover, we indicate that $u^n $ is also uniform H\"{o}lder estimate on $\mathcal{C}$. Indeed, we have the following uniform estimate on $\mathcal{C}$: for $t \ges t'$,
\begin{align}
\label{u-u'}
\ds  |u^n (t,x) - u^n (t',x')| & =|Y_t^{n,t,x}-Y_{t'}^{n,t',x'}| = \Big| \mathbb{E} \[ Y_t^{n,t,x} - Y_{t'}^{n,t',x'} \] \Big|  \nonumber \\
\ns\ds     & \les \Big|\mathbb{E} \[ Y_t^{n,t,x}-Y_{t}^{n,t',x'} \] \Big| + \Big| \mathbb{E} \[ Y_t^{n,t',x'}-Y_{t'}^{n,t',x'} \] \Big|
\end{align}
For the first term in the right-hand side  of (\ref{u-u'}), from the priori estimate in  \autoref{BSDEguji}, we have
\begin{align}
\label{inequality}
\Big|\mathbb{E} \[ Y_t^{n,t,x}-Y_{t}^{n,t',x'} \] \Big| & \les \mathbb{E} \[ |Y_t^{n,t,x}-Y_{t}^{n,t',x'}| \]
\les \|Y_t^{n,t,x}-Y_{t}^{n,t',x'}\|_{S_\dbF^2(0,T;\dbR)}  \nonumber \\
            & \les C_2  [|x-x'| + (1+|x|+|x'|)|t-t'|^\frac{1}{2}],
\end{align}
where the constant $C_2$ does not depend on $n$.
For the second term in the right hand side of (\ref{u-u'}), we consider the following BSDE:
\begin{equation*}
Y_{t'}^{n,t',x'} = Y_t^{n,t',x'} + \int^t_{t'} f_n(r,Y_r^{n,t',x'},Z_r^{n,t',x'})dr - \int^t_{t'} Z_r^{n,t',x'} dW_r.
\end{equation*}
By the growth condition of $f_n$ as in \autoref{BSDEfbijin}, we deduce that
\begin{align*}
\ds \Big| \mathbb{E} \[ Y_t^{n,t',x'}-Y_{t'}^{n,t',x'} \]\Big| & = \Big| \mathbb{E} \[ \int^t_{t'} f_n(r,Y_r^{n,t',x'},Z_r^{n,t',x'})dr \] \Big|  \\
\ns\ds            & \les \mathbb{E} \[ \int^t_{t'}  |f_n(r,Y_r^{n,t',x'},Z_r^{n,t',x'})|dr \] \\
\ns\ds            & \les \mathbb{E} \[ \int^t_{t'}  L(1 + |Y_r^{n,t',x'}| + 2|Z_r^{n,t',x'}|^2) dr \].
\end{align*}
Besides, \autoref{BSDEbijin} implies that $Y^{n,t',x'}$ and $Z^{n,t',x'}$  are bounded by $M$, so that we have
\begin{align*}
\ds \Big| \mathbb{E} \[ Y_t^{n,t',x'}-Y_{t'}^{n,t',x'} \] \Big|
& \les \mathbb{E} \[ \int^t_{t'}  L(1 + |Y_r^{n,t',x'}| + 2|Z_r^{n,t',x'}|^2) dr \]  \\
\ns\ds & \les L(1+M+2M^2)|t-t'|.
\end{align*}
Combining the inequality (\ref{inequality}) with the previous one, we easily derive that, $u^n $ admits the following uniform H\"{o}lder estimate on the compact set $\mathcal{C} \subseteq [0,T]\times \IR^m$,
\begin{align*}
\ds  |u^n (t,x) - u^n (t',x')| & =|Y_t^{n,t,x}-Y_{t'}^{n,t',x'}| \\
\ns\ds     & \les \big|\mathbb{E} \[ Y_t^{n,t,x}-Y_{t}^{n,t',x'} \] \big| + \big| \mathbb{E} \[ Y_t^{n,t',x'}-Y_{t'}^{n,t',x'} \] \big| \\
 \ns\ds    & \les C  [|x-x'| + (1+|x|+|x'|)|t-t'|^\frac{1}{2}],
\end{align*}
where the constant $C$ is independent of the parameters $n$, $t$, $t'$, $x$ and $x'$.

\ms

Finally, applying the Arzel\'{a}-Ascoli theorem on the compact set $\mathcal{C}$ to extract a subsequence of $\{ u^n \}$ which converges uniformly. Besides, noting $u^n$ increasingly converges to $\underline{u}$, we have $u^n $ uniformly converges to $\underline{u}$ on $\mathcal{C}$.
\end{proof}

\ms

Now, before present the main result of this section, we introduce the stability property of viscosity solutions below, whose proof can be fined in Lemma 6.2 of Fleming and Soner \cite{ref19}.

\begin{lemma}[Stability]\label{Stability}
Let $u^n$ be a viscosity subsolution (resp. supersolution) to the following PDE
\begin{equation*}
\frac{\partial u^{n}}{\partial t} \left( t,x\right) + H_n(t,x,u^n(t,x), \nabla u^n(t,x), D^2_{x}u^n(t,x) )=0,\quad (t,x)\in [0,T)\times \mathbb{R}^m,
\end{equation*}
where $H_n(t,x,r,p,A):[0,T]\times \mathbb{R}^m  \times \mathbb{R} \times \mathbb{R}^m \times \mathbb{S}^{m\times m} \rightarrow \mathbb{R}$ is continuous
and satisfies the ellipticity condition
\begin{equation*}
H_n(t,x,r,p,A) \les H_n(t,x,r,p,B),  \q~ \text{whenever }~A \les B.
\end{equation*}
Assume that $H_n$ and $u^n$ converge to $H$ and $u$, respectively, uniformly in every compact subset of their own domains. Then $u$ is a viscosity subsolution (resp. supersolution)  of the limit equation
\begin{equation*}
\frac{\partial u}{\partial t} \left( t,x\right)+H(t,x,u(t,x), \nabla u(t,x), D^2_{x}u(t,x) )=0,
\end{equation*}
where $D^2_{x} u$ denotes the Hessian matrix of $u$ in spatial variable $x$.
\end{lemma}

\begin{theorem}\label{21.10.5}
Under Assumptions (A1)-(A3) and (A5), for any $(t,x)\in[0,T)\times \IR^m$, $\ul{u}$ defined in (\ref{u-}) is a viscosity solution of PDE \eqref{PDE}.
\end{theorem}

\begin{proof}
On the one hand, by \autoref{BSDEfbijin}, we already have that $f_n(t,y,z)$ increasingly converges to the continuous function $f(t,y,z)$ as $n\to\infty$, and moreover, combining  Dini's theorem, $f_n(t,y,z)\rightarrow f(t,y,z)$ uniformly converges in compact subset of their domains.
On the other hand, for PDE \eqref{PDEjieduan}, observe that the infinitesimal semi-group generator $\mathcal{L}_{t,x}$ satisfies the ellipticity condition, so by \autoref{PDEzuixiaojielianxu}, we have that $u_n(t,x)$ uniformly converges to $\ul{u} (t,x)$ on the compact subset of their domains.
Finally, we apply the \autoref{Stability} to prove that $\ul{u}$ is a viscosity solution to the PDE \eqref{PDE}.
\end{proof}

\section{Large deviation principle}

\ms

Based on the above bedding, in this section, we consider the large deviation principle of the solution of backward stochastic differential equations with quadratic growth. Before going further, we consider the small perturbation of forward and backward stochastic differential equations.

\subsection{SDEs perturbation and some results}

\ms

Consider the perturbations of the following forward stochastic differential equation:
\begin{equation}\label{FBSDE:SDE}
   X^{\varepsilon,t,x}_{s} = x +\int_t^s b\left(r, X^{\varepsilon,t,x}_{r} \right) dr + \varepsilon \int_t^s
  \sigma(r) dW_r, \q~ 0 \les t \les s \les T, \q~ \varepsilon > 0;
\end{equation}
and backward stochastic differential equation:
\begin{equation}\label{FBSDE:BSDE}
  Y^{\varepsilon ,t,x}_s =  g\left( X^{\varepsilon,t,x}_T \right) + \int_s^T f\left( r,
  Y^{\varepsilon ,t,x}_r,Z^{\varepsilon ,t,x}_r \right) dr - \int_s^T Z^{\varepsilon ,t,x}_r dW_r,\q~
  0 \les t \les s \les T .
\end{equation}
We assume that the coefficients $b$ and $\sigma$ of SDE \rf{FBSDE:SDE} satisfy Assumption (A1), and $f$ and $g$ of BSDE \rf{FBSDE:BSDE} satisfy Assumptions (A2), (A3) and (A5), unless otherwise stated.

\ms

First, we would like to introduce two definitions of the large deviation, which comes from Dembo and Zeitouni \cite{ref20}.

\begin{definition} \sl
Let $E$ be a topological space.
\begin{enumerate}
  \item A rate function is a lower semicontinuous mapping ${ I}:E\rightarrow \left[ 0,+\infty \right]$, i.e., for all $a \ges 0$, the
  level set $C_I\left(a\right) =\left\{ x\in E:I\left( x\right) \les a\right\}$ is a closed subset of $E$.
  \item A good rate function is a rate function for which all the level sets $C_I\left(a\right)$ are compact subsets of $E$.
\end{enumerate}
\end{definition}
\begin{definition} \sl
The family of processes $(Y^{\varepsilon}_t)_{t\in[0,T]}$ depending on a parameter $\varepsilon$ is said to satisfy a large deviations principle with rate function $ I$ if the following conditions hold for every Borel set $A\subset C([0,T],\IR)$:
\begin{align*}
\limsup_{\varepsilon\to0} \varepsilon^2 \log\IP[Y^\varepsilon\in A]\les&~ - \inf_{\phi\in \text{Cl}(A)}  I(\phi),\\
\liminf_{\varepsilon\to0} \varepsilon^2 \log\IP[Y^\varepsilon\in A]\ges&~ - \inf_{\phi\in \text{Int}(A)}  I(\phi),
\end{align*}
where Cl\ \!$(A)$ is the closure of the set $A$ and Int\ \!$(A)$ is the interior of the set $A$.
\end{definition}

\ms

The following lemma is the well-known Freidlin-Wentzell theory, whose proof can be found in Freidlin and Wentzell \cite{ref2} or Boue and Dupuis \cite{ref29}.
\begin{lemma}
\label{Freidlin-Wentzell}
Let Assumption (A1) hold. For $(t,x)\in[0,T)\times \IR^m$,  when $\varepsilon$ goes to $0$, the solution $X^{\varepsilon,t,x}$ of SDE (\ref{FBSDE:SDE}) satisfies a large deviation principle in the space $C\left([t,T];\IR^m\right)$ associated to the rate function $I_x$, where for any function $\phi \in C\left([t,T];\IR^m\right)$,
\begin{align*}
I_x(\phi) \deq&~ \inf\Big\{ \frac{1}{2}\int_t^T |\dot{v}_r|^2 dr: \  v \in H([t,T];\IR^d)\  \text{such that } \\
&~ \phi_s =x+\int_t^s b(r,\phi_r)dr + \int_t^s \sigma(r) \dot{v}_r dr,\q~ \forall s \in [t,T]\Big\},
\end{align*}
with the convention that $\inf \varnothing =+\infty$.
\end{lemma}

\ms

Thanks to the Freidlin-Wentzell theory, we known that, when $\varepsilon$ goes to $0$, the solution $X^{\varepsilon,t,x}$ of SDE (\ref{FBSDE:SDE}) converges to the solution $\varphi^{t,x}$ of the following forward ordinary differential  equations (ODEs):
\begin{equation}\label{ODE}
\left\{\begin{aligned}
  \dot{\varphi}_s =&~ b\left( s, \varphi_s \right),\q~0 \les t \les s \les T, \\
  \varphi_t =&~ x,
\end{aligned}\right.\end{equation}
which satisfies a large deviation principle too.
Next, in order to establish some estimates of $X^{\varepsilon,t,x}$, we consider the following equation:
\begin{equation}
\label{FBSDE:SDE'}
  X^{\varepsilon',t',x'}_{s}  = x' +\int_{t'}^s b\left(r, X^{\varepsilon',t',x'}_{r} \right) dr + \varepsilon' \int_{t'}^s
  \sigma(r) dW_r, \q 0 \les t' \les s \les T, \q \varepsilon' > 0.
\end{equation}

\begin{proposition}\label{SDE'shedongguji}
Assume that Assumption (A1) holds. For any $(t,x)$, $(t',x') \in[0,T) \times \IR^m$ and $\varepsilon$, $\varepsilon' \in[0,1]$, there exists a unique solution $X^{\varepsilon,t,x}$, $X^{\varepsilon',t',x'} \in S_\dbF^2(0,T;\dbR^m) $ for the SDE \eqref{FBSDE:SDE} and \eqref{FBSDE:SDE'}, respectively. Moreover, the Malliavin derivative of $X^{\varepsilon,t,x}$ is bounded and there is a positive constant $C>0$, independent of $t, t', x, x', \varepsilon$ and $\varepsilon' $, such that
\begin{equation*}
{\Bbb E}\left[ \sup\limits_{0 \les s \les T}\left| X^{\varepsilon,t,x}_s -X^{\varepsilon',t',x'}_s \right|^2 \right]
\les  C (|t-t'| + |\varepsilon - \varepsilon'|^2 +|x-x'|^2).
\end{equation*}
\end{proposition}

\begin{proof}
Without loss of generality, we let $t \les t' \les s$. It is easy to see that under (A1), both SDE \eqref{FBSDE:SDE} and SDE \eqref{FBSDE:SDE'} admit a unique solution. Furthermore, noting that $\varepsilon \in[0,1]$, by Lemma 4.2 of Cheridito and Nam \cite{ref35}, the derivative of $X^{\varepsilon,t,x} $ exists and is bounded by some constant, which does not depend on $\varepsilon$, $t $ and $x$.

\ms

Now we give the priori estimates of the spread between the solutions $X^{\varepsilon,t,x}$ and $X^{\varepsilon',t',x'} $:
\begin{equation}\label{x-x'}
\begin{aligned}
&{\Bbb E}\left[ \sup\limits_{t'\les s\les T} \Big| X^{\varepsilon,t,x}_s -X^{\varepsilon',t',x'}_s \Big|^2 \right]\\
&~  \les 4 {\Bbb E} \bigg[ \Big|\int_t^{t'} b\left(r, X^{\varepsilon,t,x}_{r} \right) dr + \varepsilon \int_t^{t'} \sigma(r) dW_r \Big|^2
+   \sup\limits_{t'\les s\les T}\Big|\int_{t'}^s [b(r, X^{\varepsilon,t,x}_{r})  - b(r, X^{\varepsilon',t',x'}_{r})] dr\Big|^2   \\
&\q\ +   \sup\limits_{t'\les s\les T}\Big|\varepsilon \int_{t'}^s \sigma(r) dW_r - \varepsilon' \int_{t'}^s
\sigma(r) dW_r \Big|^2   +  |x-x'|^2  \bigg].
\end{aligned}
\end{equation}
For the first term in the right hand side of (\ref{x-x'}), note that $b$ and $\sigma$ are bounded by $L$ and Ito's isometry, we have
\begin{equation}\label{inequality1}
\begin{aligned}
\ds &  {\Bbb E}\bigg[ \Big|\int_t^{t'} b\left(r, X^{\varepsilon,t,x}_{r} \right) dr + \varepsilon \int_t^{t'} \sigma(r) dW_r \Big|^2 \bigg]\\
\ns\ds & \les 2 {\Bbb E} \left[ \Big|\int_t^{t'} b\left(r, X^{\varepsilon,t,x}_{r} \right) dr\Big|^2 \right]
  + 2 {\Bbb E}\left[ \Big |\varepsilon \int_t^{t'} \sigma(r) dW_r \Big|^2 \right] \\
\ns\ds    & \les 2TL^2|t'-t|   + 2 {\Bbb E} \int_t^{t'} |\varepsilon \sigma(r)|^2 dr  \\
\ns\ds    & \les   2TL^2|t'-t|   + 2T \varepsilon^2 L^2 |t'-t|.
\end{aligned}
\end{equation}
For the second term in the right hand side of (\ref{x-x'}), using the Lipschitz condition of $b$ and H\"{o}lder inequality, we have
\begin{equation}\label{inequality2}
\begin{aligned}
&{\Bbb E}\left[ \sup\limits_{t'\les s\les T}\Big|\int_{t'}^s \big(b(r, X^{\varepsilon,t,x}_{r})  - b(r, X^{\varepsilon',t',x'}_{r})\big) dr\Big|^2 \right]   \\
& \les T {\Bbb E} \int_{t'}^T L^2 \big |X^{\varepsilon,t,x}_{r}-X^{\varepsilon',t',x'}_{r} \big |^2 dr   \\
&  \les TL^2 \int_{t'}^T    {\Bbb E}\Big[ \sup\limits_{r \les s\les T} \big | X^{\varepsilon,t,x}_s -X^{\varepsilon',t',x'}_s \big |^2\Big] dr.
\end{aligned}
\end{equation}
For the third term in the right hand side of (\ref{x-x'}), using Assumption (A1), the Burkholder-Davis-Gundy inequality, we have
\begin{equation}\label{inequality3}
\begin{aligned}
& {\Bbb E}  \left[ \sup\limits_{t'\les s\les T}  \big |\varepsilon \int_{t'}^s \sigma(r) dW_r - \varepsilon' \int_{t'}^s
\sigma(r) dW_r \big |^2 \right]  \\
& \les C_1 {\Bbb E} \left[\int_{t'}^T \big |\varepsilon \sigma(r)  - \varepsilon' \sigma(r) \big |^2 dr \right]    \\
& \les C_1 TL^2 \big |\varepsilon - \varepsilon' \big|^2.
\end{aligned}
\end{equation}
Combining \eqref{inequality1}, \eqref{inequality2} and \eqref{inequality3}, noting that $0 \les \varepsilon \les 1$, we have
\begin{align*}
&  {\Bbb E}\left[ \sup\limits_{t' \les s \les T}\left| X^{\varepsilon,t,x}_s -X^{\varepsilon',t',x'}_s \right|^2 \right]    \\
&~  \les 4  {\Bbb E} \bigg[ \Big|\int_t^{t'} b\left(r, X^{\varepsilon,t,x}_{r} \right) dr + \varepsilon \int_t^{t'} \sigma(r) dW_r \Big|^2 +   \sup\limits_{t'\les s\les T} \Big|\int_{t'}^s \left[b(r, X^{\varepsilon,t,x}_{r})  - b(r, X^{\varepsilon',t',x'}_{r})\right] dr \Big|^2     \\
&\q\ +   \sup\limits_{t'\les s\les T} \Big |\varepsilon \int_{t'}^s \sigma(r) dW_r - \varepsilon' \int_{t'}^s \sigma(r) dW_r \Big|^2
 +   |x-x'|^2  \bigg]  \\
&~  \les   4 \bigg[ (2TL^2 + 2T \varepsilon^2 L^2) |t'-t| +  C_1 TL^2|\varepsilon - \varepsilon'|^2
+ TL^2  \int_{t'}^T {\Bbb E}  \left[ \sup\limits_{r \les s\les T} | X^{\varepsilon,t,x}_s -X^{\varepsilon',t',x'}_s |^2 \right] dr +|x-x'|^2 \bigg] \\
&~  \les   4\bigg[ |x-x'|^2 +(2TL^2 + 2T  L^2) |t'-t| +  C_1 TL^2|\varepsilon - \varepsilon'|^2
+ TL^2 \int_{t'}^T {\Bbb E}  \left[ \sup\limits_{r \les s\les T} | X^{\varepsilon,t,x}_s -X^{\varepsilon',t',x'}_s |^2 \right] dr
 \bigg].
\end{align*}
Then by the Gronwall inequality, we have
\begin{equation*}
{\Bbb E}\left[ \sup\limits_{t' \les s \les T}\left| X^{\varepsilon,t,x}_s -X^{\varepsilon',t',x'}_s \right|^2 \right] \les
  C (|t-t'| + |\varepsilon - \varepsilon'|^2 +|x-x'|^2).
\end{equation*}
where $C$ does not depend on $t, t', x, x', \varepsilon$ and $\varepsilon' $.
\end{proof}

\ms

Noting that when $\varepsilon = 0$, we have $X^{0,t,x} \equiv \varphi ^{t,x}$, furthermore, by \autoref{SDE'shedongguji}, the following result can be proved.

\begin{proposition}
Under Assumption (A1), take $(t,x)\in[0,T)\times \IR^m$ and for any $\varepsilon\in(0,1]$, the equation \eqref{FBSDE:SDE} and \eqref{ODE} have unique solution, $X^{\varepsilon,t,x} \in S_\dbF^2(0,T;\dbR^m) $ and $\varphi^{t,x}\in C([t,T],\IR^m)$, respectively. Moreover, there exists a constant $C>0$, independent of $t, x$ and $\varepsilon $, such that
\begin{equation*}
{\Bbb E}\left[ \sup\limits_{t \les s \les T}\left| X^{\varepsilon,t,x}_s -\varphi ^{t,x}_s \right|^2 \right]
\les C\varepsilon^2 .
\end{equation*}
\end{proposition}

\subsection{Large Deviation Principle for BSDEs}

\ms

In this subsection we give the large deviation principle of the law of $\uly^{\varepsilon,t,x}$ in the space $C\left([t,T];\IR\right)$,
namely the solution $(\uly^{\varepsilon,t,x}, \ulz^{\varepsilon,t,x})$ for BSDE (\ref{FBSDE:BSDE}) converges to, as $\varepsilon$ goes to $0$, the solution $(\psi^{t,x},0)$ of the following backward ODE:
\begin{equation}
\label{BODE}
  \dot{\psi}^{t,x}_s = - f\left( t, \psi^{t,x}_s ,0 \right), \q~ 0 \les t \les s \les T, \ \ and  \  \  \psi_T = g(\varphi^{t,x}_T).
\end{equation}

\ms

Now we consider the following two equations:
\begin{align}
\label{FBSDE:BSDEjieduan}
Y^{n,\varepsilon ,t,x}_s & =  g( X^{\varepsilon,t,x}_T ) + \int_s^T f_n ( r,
Y^{n,\varepsilon ,t,x}_r,Z^{n,\varepsilon ,t,x}_r ) dr - \int_s^T Z^{n,\varepsilon ,t,x}_r dW_r ;   \\
\label{BODEjieduan}
\dot{\psi}^{n,t,x}_s  & =  - f_n ( t,   \psi^{n,t,x}_s ,0 )  \  \ and  \ \ \psi^{n,t,x}_T = g(\varphi^{t,x}_T), \q~ 0 \les t \les s \les T,
\end{align}
where the function $f_n(t,y,z)$ is defined by (\ref{f-n}).

\ms

Obviously, for any $n \ges 2L$, by \autoref{BSDEfbijin} and the standard ODEs theory, there exists a unique continuous solution $\psi^{n,t,x}$ to the backward ODE \eqref{BODEjieduan}. And we define
 \begin{equation}
\label{u_n0}
u^{n,0} (t,x)  \deq \psi^{n,t,x}_t.
\end{equation}

\ms

Based on  \autoref{BSDEgai} and \autoref{PDEnian}, we can get the following result.
\begin{corollary}
Under Assumptions (A1)-(A3), take $(t,x)\in[0,T)\times \IR^m$ and any $\varepsilon \in(0,1]$, then for any sufficiently large $n \in  \dbN$, the function $u^{n,\varepsilon}$ defined as below:
\begin{equation}
\label{u_n}
u^{n,\varepsilon} (t,x) \deq Y^{n,\varepsilon,t,x}_t,
\end{equation}
 where $Y^{n,\varepsilon,t,x}_s$ is the first component of the solution of BSDE \eqref{FBSDE:BSDEjieduan}, is a viscosity solution of the following PDE:
\begin{equation*}\left\{\begin{aligned}
\ds &\frac{\partial u^{n,\varepsilon}}{\partial t} \left( t,x\right) + \mathcal{L}^{\varepsilon}_{t,x} u^{n,\varepsilon} \left(t,x \right)+
  f_n \left(t,u^{n,\varepsilon} \left(t,x \right),\left(\varepsilon \sigma^T \nabla u^{n,\varepsilon} \right) \left(t,x \right)\right)=0, \\
\ns\ds &u^{n,\varepsilon} \left(T,x \right)= g(x),
\end{aligned}\right.
\end{equation*}
with the second order differential operator $\mathcal{L}^{\varepsilon}_{t,x}$ being the infinitesimal semi-group generator of the Markov process
$(X_s^{\varepsilon,t,x})_{s\in[t,T]}$, the solution of SDE \eqref{FBSDE:SDE}, given by
$$
{\cal L}^{\varepsilon}_{t,x}\deq\frac{\varepsilon^2}{2} \sum_{i,j=1}^m (\sigma \sigma^T)_{i,j}\left( t\right) \frac{\partial
^2}{\partial x_i\partial x_j}+\sum_{i=1}^m b_i\left( t,x\right) \frac
\partial {\partial x_i},
$$
and $f_n(t,y,z)$ is defined by (\ref{f-n}).
\end{corollary}

\ms

In order to do establish some estimates on $Y^{n,\varepsilon,t,x}$, for any sufficiently large $n \in  \dbN$, the function $f_n(t,x,y,z)$ is defined by (\ref{f-n}), we consider the following BSDEs:
\begin{equation}
\label{FBSDE:BSDEjieduan'}
Y^{n,\varepsilon' ,t',x'}_s   =  g( X^{\varepsilon',t',x'}_T ) + \int_s^T f_n ( r,
Y^{n,\varepsilon',t',x'}_r,Z^{n,\varepsilon' ,t',x'}_r ) dr - \int_s^T Z^{n,\varepsilon' ,t',x'}_r dW_r.
\end{equation}

\begin{proposition}\label{yn-yn}
Under Assumptions (A1)-(A3) and (A5), take $(t,x)$ and $(t',x') \in[0,T) \times \IR^m$, for any $\varepsilon$, $\varepsilon' \in[0,1]$ and for any sufficiently large $n \in  \dbN$, there exists the unique solutions $(Y^{n,\varepsilon ,t,x}, Z^{n,\varepsilon ,t,x})$ and $(Y^{n,\varepsilon' ,t',x'}, Z^{n,\varepsilon' ,t',x'}) \in S_\dbF^\i(0,T;\dbR) \times L_\dbF^\infty(0,T;\dbR^d)$ for the BSDEs \eqref{FBSDE:BSDEjieduan} and \eqref{FBSDE:BSDEjieduan'}, respectively. Moreover, there exists a constant $C>0$, independent of $t, t', x, x', \varepsilon, \varepsilon'$ and $n $, such that
\begin{equation*}
\|Y_s^{n,\varepsilon ,t,x}-Y_s^{n,\varepsilon' ,t',x'}\|_{S_\dbF^2(0,T;\dbR)}^2
\les  C (|t-t'| + |\varepsilon - \varepsilon'|^2 +|x-x'|^2).
\end{equation*}
\end{proposition}

\begin{proof}
Applying the similar arguments in  \autoref{BSDEbijin}, we can easily get that there exists a unique bounded solution for the BSDEs
\eqref{FBSDE:BSDEjieduan} and \eqref{FBSDE:BSDEjieduan'}, respectively.

\ms

For any sufficiently large $n \in  \dbN$, the function $f_n(t,x,y,z)$ is defined by (\ref{f-n}), by  \autoref{BSDEfbijin},  \autoref{BSDEguji} and \autoref{SDE'shedongguji}, we have
\begin{align*}
  \|Y_s^{\varepsilon ,n,t,x}-Y_s^{n,\varepsilon' ,t',x'}\|_{S_\dbF^2(0,T;\dbR)}^2 & \les  C  \|g( X^{\varepsilon,t,x}_T ) - g( X^{\varepsilon',t',x'}_T ) \|_{L^2}^2    \\
& \les  C (|t-t'| + |\varepsilon - \varepsilon'|^2 +|x-x'|^2),
\end{align*}
where $C$ is a universal constant which will change from line to line and does not depend on $t, t', x, x', \varepsilon, \varepsilon'$ and $n $.
\end{proof}

\ms

By using the similar arguments in \autoref{BSDEbijin}, \autoref{BSDEguji} and \autoref{yn-yn}, we can get the following properties.

\begin{proposition}\label{BSDEshedongguji}
Under Assumptions (A1)-(A3) and (A5), for any $\varepsilon\in(0,1]$ and take $(t,x)\in[0,T)\times \IR^m$, for any sufficiently large $n \in  \dbN$, the equations \eqref{FBSDE:BSDEjieduan} and \eqref{BODEjieduan} have unique solutions $(Y^{n,\varepsilon,t,x},Z^{n,\varepsilon,t,x}) \in S_\dbF^\i(0,T;\dbR) \times L_\dbF^\infty(0,T;\dbR^d)$ and $\psi^{n,t,x}\in C([t,T],\IR)$, respectively, and there exists a constant $C>0$, independent of $t,$ $x,$ $\varepsilon$ and $n $, such that
\begin{equation*}
{\Bbb E}\left[ \sup\limits_{t\les s\les T}\left| Y^{n,\varepsilon,t,x}_s - \psi^{n,t,x}_s \right|^2 \right]
\les C \varepsilon^2.
\end{equation*}
Moreover, $Y^{n,\varepsilon,t,x}$ increasingly converges to $\uly^{\varepsilon,t,x}$, where the solution $(\uly^{\varepsilon,t,x}, \ulz^{\varepsilon,t,x})$ for BSDE \eqref{FBSDE:BSDE}, as $n\rightarrow \infty$.
\end{proposition}

\begin{proof}
For any $\varepsilon\in(0,1]$ and take $(t,x)\in[0,T)\times \IR^m$, for any sufficiently large $n \in  \dbN$, by \autoref{BSDEfbijin} and applying the similar arguments as \autoref{BSDEbijin}, there exists a unique solution $(Y^{n,\varepsilon,t,x}, Z^{n,\varepsilon,t,x})$ to the equation \eqref{FBSDE:BSDEjieduan}. By the theory of ODEs, the existence and uniqueness solution of the equation \eqref{BODEjieduan} is also easily proved.

\ms

Noting that when $\varepsilon = 0$, $X^{0,t,x} \equiv \varphi ^{t,x}$ and $Y^{n,0,t,x} \equiv \psi^{n,t,x}$, then by Proposition \ref{yn-yn}, we have
\begin{align*}
  {\Bbb E}\left[ \sup\limits_{t\les s\les T}\left| Y^{n,\varepsilon,t,x}_s - \psi^{n,t,x}_s \right|^2 \right] \les C \varepsilon^2.
\end{align*}

\ms

Moreover, by \autoref{BSDEbijin}, we have $Y^{n,\varepsilon,t,x}$ is monotone increasing convergence to $\uly^{\varepsilon,t,x}$, as $n\rightarrow \infty$.
\end{proof}

\ms

Let's recall the link between FBSDEs and PDEs in the viscosity sense, namely in view of \autoref{21.10.5} by solving the small perturbation FBSDE \eqref{FBSDE:SDE}-\eqref{FBSDE:BSDE} one is able to solve the following PDEs, for $(t,x)\in[0,T] \times \IR^m ~\text{and} ~ \varepsilon >0$,
\begin{equation}\label{PDEshedong}\left\{\begin{aligned}
\ds &\frac{\partial u^\varepsilon}{\partial t} \left( t,x\right) + \mathcal{L}^{\varepsilon}_{t,x} u^\varepsilon \left(t,x \right)+
  f\left(t,u^\varepsilon \left(t,x \right),\left(\varepsilon \sigma^T \nabla u^\varepsilon \right) \left(t,x \right)\right)=0, \\
\ns\ds &u^\varepsilon \left(T,x \right) = g(x),
\end{aligned}\right.
\end{equation}

\ms

So the relation between the first component $\ul{Y}^{\varepsilon,t,x}_t$ of the solution to BSDE \eqref{FBSDE:BSDE} and the viscosity solution $\ul{u}^\varepsilon(t,x)$, is defined by
\begin{equation}\label{ue}
  \ul{u}^\varepsilon(t,x) \deq \uly^{\varepsilon,t,x}_t,
\end{equation}
to the PDE \eqref{PDEshedong} is given by the identity:
$$\ul{Y}^{\varepsilon,t,x}_s = \ul{u}^\varepsilon(s,X^{\varepsilon,t,x}_s), \q~ s \in [t,T].$$
And when $\varepsilon =0$ one is lead to the first order PDEs for $(t,x)\in[0,T]\times\IR^m$
\begin{equation}\label{PDEyijie}\left\{\begin{aligned}
\ds &\frac{\partial u^0}{\partial t} \left( t,x\right)+ (b^T \nabla u^0)(t,x)
 + f\big( t,u^0(t,x),0 \big)=0, \\
\ns\ds &u^0(T,x)=g(x).
\end{aligned}\right.
\end{equation}
We can also get the relationship between the first component of the solution $(\psi^{t,x},0)$ of the backward ODE \eqref{BODE} and
the viscosity solution $\ul{u}^0(t,x)$, is defined by
\begin{equation}\label{u0}
  \ul{u}^0 (t,x) \deq \psi^{t,x}_t,
\end{equation}
to the PDE \eqref{PDEyijie}, is given by the identity:
$$\psi^{t,x}_s = \ul{u}^0(s,\varphi^{t,x}_s), \q~ s \in [t,T].$$

\ms

Notice that we can always interpret the backward ODEs as the deterministic BSDEs. This interpretation allows us to use \autoref{yn-yn} and by similar arguments in \autoref{PDEzuixiaojielianxu} to obtain the function $u^{n,\varepsilon}(t,x)$ uniformly converges to $\ul{u}^\varepsilon(t,x)$, with respect to $\varepsilon, t$ and $x$, as $n \rightarrow \infty.$

\begin{proposition}\label{PDEshedongzuixiaojielianxu}
Under Assumptions (A1)-(A3) and (A5), then for any sufficiently large $n \in  \dbN$, the function $u^{n,\varepsilon}(t,x)$ defined in \eqref{u_n} and \eqref{u_n0} converges to $\ul{u}^\varepsilon(t,x)$ defined in \eqref{ue} and \eqref{u0}, uniformly in any compact subset of $[0,1]\times [0,T]\times \IR^m$, as $n \rightarrow \infty.$
\end{proposition}

\begin{proof}
For any sufficiently large $n \in  \dbN$, we consider the perturbed FBSDE \eqref{FBSDE:SDE}-\eqref{FBSDE:BSDEjieduan}. Denote by $(Y^{F^1},0)$ and $(Y^{F^2},0)$ the solutions of two deterministic BSDEs with the generators
$F^1(r,y,z)= L(1+|y|+2|z|^2)$ and $F^2(r,y,z)= - L(1+|y|+2|z|^2)$, the terminal conditions $L$ and $-L$,
respectively. By \autoref{BSDEfbijin} and the comparison theorem (Theorem 7.3.1 of Zhang \cite{ref26}), we get
$Y^{F^2}_t \les Y^{n,\varepsilon ,t,x}_t \les Y^{F^1}_t$.
Since the solutions $Y^{F^1}_t$ and $Y^{F^1}_t$ are bounded on $[0,T]$, we have that for any compact set $\mathcal{C}$ of $[0,1]\times [0,T]\times \IR^m$, the sequence $\{ u^{n,\varepsilon}(t,x) \} _n$ defined in \eqref{u_n} and \eqref{u_n0} is uniformly bounded on $\mathcal{C}$.

\ms

Now we indicate that $\{ u^{n,\varepsilon}(t,x) \} _n$ is uniformly H\"{o}lder estimate on $\mathcal{C}$ too. Indeed, we have the following uniform estimate on $\mathcal{C}$: For $t \ges t'$,
\begin{align}
\label{u-u'e}
\ds  |u^{n,\varepsilon} (t,x) - u^{n,\varepsilon'} (t',x')| & =|Y_t^{n,\varepsilon,t,x}-Y_{t'}^{n,\varepsilon',t',x'}| = \Big| \mathbb{E} \[ Y_t^{n,\varepsilon,t,x}-Y_{t'}^{n,\varepsilon',t',x'} \] \Big|  \nonumber \\
\ns\ds     & \les \Big|\mathbb{E} \[ Y_t^{n,\varepsilon,t,x}-Y_{t}^{n,\varepsilon',t',x'} \] \Big| + \Big| \mathbb{E} \[ Y_t^{n,\varepsilon',t',x'}-Y_{t'}^{n,\varepsilon',t',x'} \] \Big|.
\end{align}
For the first term in the right-hand side  of (\ref{u-u'e}), from  \autoref{yn-yn}, we have
\begin{align}
\label{4inequality}
\Big|\mathbb{E} \[ Y_t^{n,\varepsilon,t,x}-Y_{t}^{n,\varepsilon',t',x'} \] \Big| & \les \mathbb{E} \[ |Y_t^{n,\varepsilon,t,x}-Y_{t}^{n,\varepsilon',t',x'}| \]
\les \|Y_t^{n,\varepsilon,t,x}-Y_{t}^{n,\varepsilon',t',x'}\|_{S_\dbF^2(0,T;\dbR)}  \nonumber \\
            & \les C_2  \left(|t-t'| + |\varepsilon - \varepsilon'|^2 +|x-x'|^2\right),
\end{align}
where the constant $C_2$ does not depend on $t, t', x, x', \varepsilon, \varepsilon'$ and $n $.
For the second term in the right hand side of (\ref{u-u'e}), we consider the following BSDE:
\begin{equation*}
Y_{t'}^{n,\varepsilon',t',x'} = Y_t^{n,\varepsilon',t',x'} + \int^t_{t'} f_n(r,Y_r^{n,\varepsilon',t',x'},Z_r^{n,\varepsilon',t',x'})dr - \int^t_{t'} Z_r^{n,\varepsilon',t',x'} dW_r.
\end{equation*}
By the growth condition of $f_n$ as in \autoref{BSDEfbijin}, we deduce that
\begin{align*}
\ds \Big| \mathbb{E} \[ Y_t^{n,\varepsilon',t',x'}-Y_{t'}^{n,\varepsilon',t',x'} \]\Big| & = \Big| \mathbb{E} \[ \int^t_{t'} f_n(r,Y_r^{n,\varepsilon',t',x'},Z_r^{n,\varepsilon',t',x'})dr \] \Big|  \\
\ns\ds            & \les \mathbb{E} \[ \int^t_{t'}  |f_n(r,Y_r^{n,\varepsilon',t',x'},Z_r^{n,\varepsilon',t',x'})|dr \] \\
\ns\ds            & \les \mathbb{E} \[ \int^t_{t'}  L(1 + |Y_r^{n,\varepsilon',t',x'}| + 2|Z_r^{n,\varepsilon',t',x'}|^2) dr \].
\end{align*}
Besides, \autoref{yn-yn} implies that $Y^{n,\varepsilon',t',x'}$ and $Z^{n,\varepsilon',t',x'}$  are bounded by some constant $M$, so that we have
\begin{align*}
\ds \Big| \mathbb{E} \[ Y_t^{n,\varepsilon',t',x'}-Y_{t'}^{n,\varepsilon',t',x'} \] \Big|
& \les \mathbb{E} \[ \int^t_{t'}  L(1 + |Y_r^{n,\varepsilon',t',x'}| + 2|Z_r^{n,\varepsilon',t',x'}|^2) dr \]  \\
\ns\ds & \les L(1+M+2M^2)|t-t'|.
\end{align*}
Combining the inequality (\ref{4inequality}) with the previous one, we easily derive that the sequence $\{ u^{n,\varepsilon}(t,x) \} _n$ admits the following uniform H\"{o}lder estimate on the compact set $\mathcal{C} \subseteq [0,1] \times [0,T]\times \IR^m$,
\begin{align*}
\ds  |u^{n,\varepsilon} (t,x) - u^{n,\varepsilon'} (t',x')|  \les C  \left(|t-t'| + |\varepsilon - \varepsilon'|^2 +|x-x'|^2\right),
\end{align*}
where the constant $C$ is independent of the parameters $n$, $t$, $t'$, $\varepsilon$, $\varepsilon'$, $x$ and $x'$.

\ms

Finally, applying the Arzel\'{a}-Ascoli theorem on the compact set $\mathcal{C}$ to extract a subsequence of $\{ u^{n,\varepsilon}(t,x) \} _n$, which is converging uniformly. Besides, note that $u^{n,\varepsilon}(t,x)$ is increasingly converges to $\ul{u}^\varepsilon(t,x)$, which completes the proof.
\end{proof}

\ms

To establish the large deviation principle for the solution of the BSDE \eqref{FBSDE:BSDE}, we need the following well-known contraction principle from Varadhan \cite{ref21}.

\begin{lemma}[Contraction principle]
\label{Contraction principle}
Let $\mathcal{X}$ and $\mathcal{Y}$ be two Polish spaces, and $\IP^{\varepsilon}$ a family of probability measures on the Borel subsets of $\mathcal{X}$. Let $\IP^{\varepsilon}$ satisfies the large deviation principle on $\mathcal{X}$ with a vrate function $I$. Let $F^\varepsilon$ be continuous maps from $\mathcal{X}$ to $\mathcal{Y}$ and assume that $\lim_{\varepsilon\to0} F^\varepsilon = F$ exists uniformly over compact subsets of $\mathcal{X}$. If we define the $\IQ^{\varepsilon}$ on $\mathcal{Y}$ by $\IQ^{\varepsilon} \deq \IP^{\varepsilon}(F^\varepsilon)^{-1}$, then $\IQ^{\varepsilon}$ satisfies the large deviation principle on $\mathcal{Y}$ with a rate function $J$ defined by
$$J(y) \deq \inf\Big\{I(x): x \in \mathcal{X} \text{ such that } y=F(x) \Big\}.$$
\end{lemma}

\begin{definition}
For any $t \in [0,T)$, the operator $F^\varepsilon$ is defined as
\begin{align*}
F^\varepsilon: C([t,T], \IR^m) & \to C([t,T],\IR)
\\
\psi & \mapsto F^\varepsilon(\psi) \deq \ul{u}^{\varepsilon}(\cdot,\psi_\cdot),
\end{align*}
where for any $\varepsilon > 0$, the function $\ul{u}^{\varepsilon}$ is defined in \eqref{ue}, and for $\varepsilon=0$, $\ul{u}^0$ is defined in \eqref{u0}.
\end{definition}

\ms

We observe that for all $0 \les t \les s \les T$, $x\in \IR^m$ and  $\varepsilon \ges 0$ we have $\ul{Y}^{\varepsilon,t,x}_s = F^\varepsilon( X^{\varepsilon,t,x})(s) = \ul{u}^\varepsilon(s, X^{\varepsilon,t,x}_s )$ . We now state and prove the main result concerning the LDP satisfied by the law induced by $\ul{Y}^{\varepsilon,t,x}$.

\begin{theorem}[Large Deviation Principle]
\label{BSDELDP}
 Let Assumptions (A1)-(A3) and (A5) hold, for any $(t,x)\in[0,T)\times \IR^m$, then the first component $\ul{Y}^{\varepsilon,t,x}$ of the solution of BSDE \eqref{FBSDE:BSDE} satisfies, as $\varepsilon$ goes to 0, a large deviation principle in the space $C\left([t,T];\IR\right)$, with the rate function $\widehat{I}_x$ defined for any function $\psi \in C\left([t,T];\IR\right)$ by
\begin{align*}
\widehat{I}_x(\psi)
& \deq
\inf\Big\{
I_x(\varphi): \varphi\in H([t,T],\IR) \text{ such that } \psi =F^0(\varphi)
\Big\},
\end{align*}
with the convention that $\inf \varnothing =+\infty$ and where $I_x(\cdot)$ is the rate function defined in \autoref{Freidlin-Wentzell}.

\end{theorem}

\begin{proof}
We will apply the well known contraction principle to prove the theorem (see \autoref{Contraction principle} or Theorem 2.4 in Varadhan \cite{ref21}). So we only need to show that  for $\varepsilon\in(0,1]$, $F^\varepsilon$ is continuous operator from $C([t,T],\IR^m)$ onto $C([t,T],\IR)$, and moreover, that $F^\varepsilon$ converges uniformly to $F^0$ over all compact sets of $C([t,T],\IR^m)$ as $\varepsilon$ vanishes.

\ms

{\sf Step 1 - Continuity of $F^\varepsilon$:} Let $\varepsilon > 0$ and $\phi \in C([t,T],\IR^m)$. We will prove that $F^\varepsilon$ is continuous at $\phi$. Let $\{\phi^n \}_n$ be a sequence in $C([t,T],\IR^m)$ which converges to $\phi$ under the uniform norm $\rho_{[t,T]}(\phi)$.
Then, there exists an constant $M >0$ such that for any $n \in \IN$,  $\rho_{[t,T]}(\phi^n) \les M$ and $\rho_{[t,T]}(\phi) \les M$. By the \autoref{21.10.5}, we known $F^\varepsilon(\phi) \deq \ul{u}^{\varepsilon}(\cdot,\phi_\cdot)$ is continuous, so $\ul{u}^{\varepsilon}$ is uniformly continuous on any compact subset of $[t,T] \times \mathcal{K}$, where $\mathcal{K}$ is the closed ball centered at the origin with
radius $M$ in $\IR^m$. Therefore, for any $\epsilon >0$, there exists $\delta > 0$ such that $|s-s'|<\delta$ and $|x-x'|<\delta,x,x' \in \mathcal{K}$ imply $|\ul{u}^{\varepsilon} (s,x) - \ul{u}^{\varepsilon} (s',x')| \les \epsilon$. Since there exists an $n_0$ such that for all
$n \ges n_0, \ \ \rho_{[t,T]}(\phi^n -\phi) \les \delta$, so for any $s \in [t,T]$ and $\phi_s,$ $\phi^{n}_s \in \mathcal{K}$  , we have $|\ul{u}^{\varepsilon} \left(s,\phi^{n}_s\right) - \ul{u}^{\varepsilon} \left(s,\phi_s\right)| \les \epsilon$, this imply that $F^\varepsilon(\phi^n)$ converges to $F^\varepsilon(\phi)$ as $n$ tend to $\infty$.

\ms

{\sf Step 2 - Convergence over compact subsets:} To prove that $F^\varepsilon$ uniformly converges to $F^0$  on every compact subset of $C([t,T],\IR^m)$. Let  $\cK^0$ be a compact subset of $C([t,T],\IR^m)$, and $\phi \in \cK^0$. We have
\begin{align}
&
\sup_{\phi\in\cK^0}\left[
\rho_{[t,T]}\big(F^{\varepsilon}(\phi)-F^{0}(\phi)\big)\right]
 \nonumber \\
&\quad =\sup_{\phi\in \cK^0} \sup_{s\in[t,T]}
|\ul{u}^{\varepsilon}(s,\phi_s)-\ul{u}^{0}(s,\phi_s)|
=\sup_{\phi\in \cK^0} \sup_{s\in[t,T]}
|\ul{Y}^{\varepsilon,s,\phi_s}_s-\psi^{s,\phi_s}_s|
 \nonumber \\
&
\quad
\les \sup_{x\in \cJ} \sup_{s\in[t,T]}
|\ul{Y}^{\varepsilon,s,x}_s-\psi^{s,x}_s|   \nonumber  \\
& \quad \les \sup_{x\in \cJ} \sup_{s\in[t,T]}
\left(|\ul{Y}^{\varepsilon,s,x}_s-Y^{n,\varepsilon,s,x}_s| + |Y^{n,\varepsilon,s,x}_s - \psi^{n,s,x}_s| + |\psi^{n,s,x}_s-\psi^{s,x}_s| \right),
\nonumber
\end{align}
where we set $\cJ \deq \{\phi_s: \phi\in\cK^0,\ s\in[t,T]\}$, as the set of points consisting of the images of the $\phi \in \cK^0$ in $\IR^m$ for all $s\in[t,T]$. Since $\phi$ is continuous, $\cJ$ is a compact set of $\IR^m$.

\ms

Then, by  \autoref{PDEshedongzuixiaojielianxu}, we have
\begin{equation}\label{l1}
  \sup_{\varepsilon \in [0,1]}\sup_{x\in \cJ} \sup_{s\in[t,T]} |\ul{Y}^{\varepsilon,s,x}_s-Y^{n,\varepsilon,s,x}_s|=\sup_{\varepsilon \in [0,1]}\sup_{x\in \cJ} \sup_{s\in[t,T]}|\ul{u}^\varepsilon(s,x)-u^{n,\varepsilon}(s,x)|
\end{equation}
going to 0, as $n$ tends to infinitely. By  \autoref{BSDEshedongguji}, there exists a constant $C>0$ which is independent of $s,$ $x,$ $\varepsilon$ and $n $, such that
\begin{equation}\label{l2}
\sup_{x\in \cJ} \sup_{s\in[t,T]} |Y^{n,\varepsilon,s,x}_s - \psi^{n,s,x}_s| \les C \varepsilon .
\end{equation}
From  \autoref{PDEshedongzuixiaojielianxu}, we also have that
\begin{equation}\label{l3}
\sup_{x\in \cJ} \sup_{s\in[t,T]} |\psi^{n,s,x}_s-\psi^{s,x}_s|=\sup_{x\in \cJ} \sup_{s\in[t,T]}|u^{n,0}(s,x)-\ul{u}^0(s,x)|
\end{equation}
going to 0, as $n$ tends to infinitely.

\ms

Combining \eqref{l1}, \eqref{l2} and \eqref{l3}, we can get
\begin{equation*}
   \lim_{\varepsilon \rightarrow 0}\lim_{n \rightarrow \infty}\sup_{\phi\in\cK^0}\left[ \rho_{[t,T]}\big(F^{\varepsilon}(\phi)-F^{0}(\phi)\big)\right]=0.
\end{equation*}
So we can get $F^\varepsilon$ uniformly converges to $F^0$  on every compact subset of $C([t,T],\IR^m)$.

\ms

The result now follows from the above mentioned contraction principle.
\end{proof}
\begin{remark}\label{3}
In fact, the rate function $\widehat{I}_x$ in  \autoref{BSDELDP} can be given following:
$$\widehat{I}_x(\psi)= \inf_{v \in \cA } \frac{1}{2}\int_t^T |\dot{v}_r|^2 dr $$
where
\begin{align*}
\cA &= \big\{ v: v \in H([t,T];\IR^d) \text{ such that } \psi_s =\psi^{t,x}_s = g(\varphi^{t,x}_T) + \int_s^T f\left( u, \psi^{t,x}_u ,0 \right) du, \\
& \hspace{1.4cm} \varphi^{t,x}_s =x+\int_t^s b(r,\varphi^{t,x}_r)dr + \int_t^s \sigma(r) \dot{v}_r dr,  \q~ 0 \les t \les s \les T \big\},
\end{align*}
with the convention that $\inf \varnothing =+\infty$.
\end{remark}

\ms



For the LDP for BSDE with the superquadratic generator, in our coming paper, we will show that our approach in this paper is flexible enough to be able to treat the superquadratic BSDE case without major modifications to our proofs (For the sup-convolution approximation technicalities, see Shi and Yang \cite{ref202} for detail).

\end{document}